\documentclass[a4paper,12pt,reqno]{amsart}
\usepackage[utf8]{inputenc}
\usepackage{amssymb}
\usepackage{amsmath}
\usepackage{makecell}
\usepackage{multirow}
\usepackage{hhline}
\usepackage{amsthm}
\usepackage{amsfonts}
\usepackage{mathtools}
\usepackage{todonotes}
\usepackage{comment}
\usepackage{mathrsfs}
\usepackage{enumitem}
\usepackage[all,cmtip]{xy}
\usepackage{ctable}
\usepackage[hyphens]{url}
\usepackage{hyperref}
\usepackage[hyphenbreaks]{breakurl}
\usepackage{cleveref}

\usepackage[british]{babel}
\usepackage[margin=1.2in]{geometry}
\usepackage{eqparbox}
\newcommand{\eqmathbox}[2][M]{\eqmakebox[#1]{$\displaystyle#2$}}
\usepackage{pbox}

\renewcommand{\O}{\mathcal{O}} 

\newcommand{\nequiv}{\not\equiv}
\newcommand{\set}[1]{\left\lbrace #1 \right\rbrace}

\newcommand{\field}[1]{\mathbb{#1}}  
\newcommand{\Q}{\field{Q}} 
\newcommand{\R}{\field{R}} 
\newcommand{\C}{\field{C}} 
\newcommand{\Z}{\field{Z}} 
\newcommand{\F}{\field{F}} 
\newcommand{\T}{\field{T}} 

\newcommand{\I}{\field{I}}

\renewcommand{\P}{\field{P}}

\newcommand{\mc}[1]{\mathcal{#1}}

\newcommand{\Mod}[1]{\ (\mathrm{mod}\ #1)}
\newcommand{\Nm}{\textup{Nm}}
\newcommand{\Tr}{\textup{Tr}}
\newcommand{\fp}{\mathfrak{p}}
\newcommand{\fm}{\mathfrak{m}}

\newcommand{\fq}{\mathfrak{q}}

\newcommand{\lcm}{\textup{lcm}}
\newcommand{\Gen}{\textup{\textsf{Gen}}}
\newcommand{\Aux}{\textup{\textsf{Aux}}}
\newcommand{\AuxGen}{\textup{\textsf{AuxGen}}}
\newcommand{\TypeTwoNotMomoseBound}{\textup{\textsf{TypeTwoNotMomoseBound}}}
\newcommand{\TypeThreeNotMomoseBound}{\textup{\textsf{TypeThreeNotMomoseBound}}}

\DeclareMathOperator{\im}{im}
\DeclareMathOperator{\id}{Id}

\DeclareMathOperator{\Frob}{Frob}
\DeclareMathOperator{\frob}{Frob}

\DeclareMathOperator{\Hom}{Hom}
\DeclareMathOperator{\Aut}{Aut}

\DeclareMathOperator{\Cl}{Cl}

\DeclareMathOperator{\Pic}{Pic}
\DeclareMathOperator{\Spec}{Spec}

\DeclareMathOperator{\Gal}{Gal}

\DeclareMathOperator{\HCF}{\textup{\textsf{HCF}}}
\DeclareMathOperator{\IsogPrimeDeg}{\textup{\textsf{IsogPrimeDeg}}}
\DeclareMathOperator{\BadFormalImmersion}{\textup{\textsf{BadFormalImmersion}}}
\DeclareMathOperator{\AGFI}{\textup{\textsf{AGFI}}}

\newcommand{\eps}{\varepsilon} 

\newcommand{\TypeTwoPrimes}{\textup{\textsf{TypeTwoPrimes}}}
\newcommand{\TypeOneBound}{\textup{\textsf{TypeOneBound}}}
\newcommand{\GenericBound}{\textup{\textsf{GenericBound}}}
\newcommand{\MMIB}{\textup{\textsf{MMIB}}}

\newtheorem{lemma}{Lemma}
\newtheorem{theorem}[lemma]{Theorem}
\newtheorem{proposition}[lemma]{Proposition}
\newtheorem{corollary}[lemma]{Corollary}
\newtheorem{algorithm}[lemma]{Algorithm}

\theoremstyle{definition}
\newtheorem{definition}[lemma]{Definition}
\newtheorem{example}[lemma]{Example}

\newtheorem{remark}[lemma]{Remark}

\newtheorem*{ack}{Acknowledgements}

\numberwithin{lemma}{section}
\numberwithin{equation}{section} 
\numberwithin{figure}{section}

\title{Explicit isogenies of prime degree over number fields}
\author{Barinder S. Banwait}
\address{Barinder S. Banwait \\
Department of Mathematics \& Statistics\\  
Boston University\\
111 Cummington Mall\\
Boston, MA 02215\\
USA}
\email{barinder.s.banwait@gmail.com}

\author{Maarten Derickx}
\address{Maarten Derickx,
Den Haag,
The Netherlands}
\email{maarten@mderickx.nl}

\date{}

\makeatletter
\providecommand\@dotsep{5}
\renewcommand{\listoftodos}[1][\@todonotes@todolistname]{%
  \@starttoc{tdo}{#1}}
\makeatother


\begin{document}

\dedicatory{Dedicated to the memory of Sebastiaan Johan Edixhoven, 1962 - 2022}

\maketitle

\begin{abstract}
We provide an explicit and algorithmic version of a theorem of Momose classifying isogenies of prime degree of elliptic curves over number fields, which we implement in Sage and PARI/GP. Combining this algorithm with recent work of Box-Gajovi\'c-Goodman we obtain the first classifications of the possible prime degree isogenies of elliptic curves over cubic number fields, as well as for several quadratic fields not previously known. While the correctness of the general algorithm relies on the Generalised Riemann Hypothesis, the algorithm is unconditional for the restricted class of semistable elliptic curves.
\end{abstract}


\section{Introduction}

Let $k$ be a number field, and consider the set $\IsogPrimeDeg(k)$ of primes $p$ which arise as $k$-rational isogenies of degree $p$ as one varies over all elliptic curves over $k$; we refer to such $p$ as \emph{isogeny primes for $k$}. The set $\IsogPrimeDeg(k)$ is necessarily infinite if $k$ contains the Hilbert class field of an imaginary quadratic field, which follows from the basic theory of complex multiplication on elliptic curves, and it is a consequence of work of Momose \cite{momose1995isogenies} with Merel's proof of uniform boundedness for torsion on elliptic curves \cite{merel1996bornes} that the converse holds assuming the Generalised Riemann Hypothesis (GRH). This consequence was written in the literature explicitly by Larson and Vaintrob.

\begin{theorem}[Larson and Vaintrob, Corollary 2 in \cite{larson_vaintrob_2014}\label{thm:isog-finite}]
Assume the Generalised Riemann Hypothesis. For a number field $k$, $\IsogPrimeDeg(k)$ is finite if and only if $k$ does not contain the Hilbert class field of an imaginary quadratic field.
\end{theorem}

The question of exactly determining $\IsogPrimeDeg(k)$ when it is finite originated with the seminal work of Mazur who determined the base case of $k = \Q$.

\begin{theorem}[Mazur, Theorem 1 in \cite{mazur1978rational}]
\[ \IsogPrimeDeg(\Q) = \left\{2, 3, 5, 7, 11, 13, 17, 19, 37, 43, 67, 163\right\}.\]
\end{theorem}

Recent work of the first named author \cite{banwait2021explicit} found the first instances of the determination of $\IsogPrimeDeg(k)$ for some quadratic fields.

\begin{theorem}[Banwait]
Assuming GRH, we have the following.
\begin{align*}
\IsogPrimeDeg(\Q(\sqrt{-10})) &= \IsogPrimeDeg(\Q),\\
\IsogPrimeDeg(\Q(\sqrt{5})) &= \IsogPrimeDeg(\Q) \cup \left\{23, 47\right\},\\
\IsogPrimeDeg(\Q(\sqrt{7})) &= \IsogPrimeDeg(\Q).
\end{align*}
\end{theorem}

More generally, an algorithm was presented in \emph{loc. cit.} which took as input a quadratic field which is not imaginary quadratic of class number one, and output a small superset for $\IsogPrimeDeg(k)$. Results on quadratic points on low-genus modular curves \cite{bruin2015hyperelliptic, ozman2019quadratic, box2021quadratic} together with work on everywhere local solubility of twists of modular curves \cite{ozman2012points} then allowed one to pass from the superset to the actual set itself. 

This algorithm relied crucially on work of Momose, who proved (Theorem A/Theorem 1 in \cite{momose1995isogenies}) that if there exists an elliptic curve admitting a $k$-rational $p$-isogeny, then for $p$ larger than a constant $C_k$ depending only on $k$, the associated isogeny character obtained from considering the Galois action on the kernel of the isogeny must be one of three defined ``Types'', hereafter referred to as \textbf{Momose Types $1$, $2$ and $3$}. We refer to this result as \emph{Momose's isogeny classification theorem}. We will recap the definition of these ``Types'' in \Cref{ssec:from_sig_types_to_momose}; for now we mention only that Momose Type 3 requires that $k$ contain the Hilbert class field of an imaginary quadratic field, and that, with this assumption, the infinitely many resulting isogeny primes from CM elliptic curves are of Momose Type 3. In principle this reduces the task of bounding isogeny primes for $k$ (when this set of isogeny primes is finite) to bounding isogeny primes arising from Momose Types $1$ and $2$. The former appears as Theorem 3 of \cite{momose1995isogenies}, and the latter - which requires GRH - as Remark 8 of \cite{momose1995isogenies}. Taken together we refer to these three results as \emph{Momose's isogeny theorems}, and we note furthermore that Momose assumed in his work that $p$ is unramified in $k$.

Momose did not make the constant $C_k$ explicit; this was done subsequently by David (Section 2.4 of \cite{david2012caractere} or Section 2.3 of \cite{david2008caractere}) in the case that $k$ is Galois over $\Q$. In the course of this work, David gave a more precise and careful treatment of the proof of Momose's Theorem 1 in the Galois case; in particular a reproof of Momose's Lemmas 1 and 2 in \cite{momose1995isogenies} (which originally assumed that $k$ is Galois over $\Q$); see Section 2 of \cite{banwait2021explicit} for an overview of David's work and for precise references to it. We note that David also assumed that $p$ is unramified in $k$.

While Momose's Lemmas 1 and 2 \emph{do} assume that $k$ is Galois over $\Q$, Momose gives an argument in the proof of Theorem 1 to cover the non-Galois case by passing to the Galois closure $K$ over $\Q$. Unfortunately, the details given there contain some mistakes and gaps, which we indicate as the first three items of the following; the last item indicates a mistake in Momose's Theorem 3.

\begin{enumerate}
    \item
    Momose defines a certain group ring character $\eps$ at the outset of the proof, but it is unclear whether he is taking this to be over $k$ or over $K$. Based on his exposition, and his definition of $d$ as $[k : \Q]$, it appears that he is taking the base field to be $k$; however he is using Lemma 1 to define $\eps$, and Lemma 1 assumes that the field is Galois over $\Q$, suggesting he is taking $K$ to be the ground field. It is not \emph{a priori} clear that such an $\eps$ exists over non-Galois number fields $k$.
    \item Later in the proof Momose takes a set of generators of the class group of $k$ consisting of completely split primes in $k$. Such a generating set only necessarily exists in the Galois case.
    \item
    Immediately after the displayed equation (2), Momose writes ``In the case of Type 2, $\beta = \zeta\sqrt{-q}$ for a $12h$\textsuperscript{th} root $\zeta$ of unity''. This would follow if one had the interpretation of ``Type 2'' as $\eps = 6\Nm_{k/\Q}$ as promised by Lemma 2. However, again, since Lemma 2 also assumes that the ground field is Galois over $\Q$, this move is not valid.
    \item
    In the proof of Theorem 3 of \cite{momose1995isogenies}, Momose concludes that $p - 1 \mid 12h_k$ in the potentially multiplicative reduction case ($h_k$ being the class number of $k$); however this bound is too strict, and should rightly be $p \mid \Nm(\fq)^{12h_k}-1$, for $\fq$ a prime of $k$ coprime to $p$. See \Cref{rem:mom_mistake_type_1} and its preceding discussion for more details.
\end{enumerate}

One of the main contributions of this paper is to offer corrected proofs of Momose's isogeny theorems which furthermore strengthen them to deal with the case that $p$ ramifies in $k$. In addition, since our chief motivation is to compute exact sets of isogeny primes, we take this opportunity to recast Momose's Theorem 1 into a more algorithmic framing, providing a generalisation of the previous algorithm of the first named author to arbitrary number fields. This algorithm may be found as \Cref{alg:main}, and it yields the following explicit version of Momose's isogeny classification theorem.

\begin{theorem}\label{thm:main-brief}
Let $k$ be a number field. Then \Cref{alg:main} computes a nonzero integer $\MMIB(k)$ such that, if $p$ is an isogeny prime for $k$ whose associated isogeny character is not of Momose Type $1$, $2$ or $3$, then $p$ divides $\MMIB(k)$.
\end{theorem}

We refer to $\MMIB(k)$ as the \textbf{Momose Multiplicative Isogeny Bound} of $k$. Contrary to Momose's passing to the Galois closure of $k$ over $\Q$, our approach will be to strengthen Momose's Lemmas 1 and 2 (or rather, David's versions of these Lemmas) to remove the Galois assumption, which furthermore provides for a favourable improvement to the algorithm (see \Cref{rem:not_passing_to_gal_closure_is_faster}).

Since isogenies arising from CM elliptic curves are necessarily of Momose Type 3, we cannot hope to bound such `Momose Type $3$ isogeny primes'. Thus, in our attempt to find a multiplicative bound on $\IsogPrimeDeg(k)$, we are reduced to bounding isogeny primes which arise from isogeny characters of Momose Types $1$ or $2$.

By building on the explicit criteria - given by work of the second named author with Kamienny, Stein and Stoll \cite{derickx2019torsion} - for when the natural map $X_0(p)^d \rightarrow J_e$ of the $d$\textsuperscript{th} symmetric power modular curve into the winding quotient of $J_0(p)$ is a formal immersion in positive characteristic - we are able to explicitly and algorithmically determine a multiplicative bound on Momose Type 1 primes.

\begin{theorem}\label{thm:type_1_brief}
Let $k$ be a number field. Then \Cref{alg:type_1_primes} computes a nonzero integer $\TypeOneBound(k)$ such that, if $p$ is an isogeny prime for $k$ whose associated isogeny character is of Momose Type $1$, then $p$ divides $\TypeOneBound(k)$.
\end{theorem}

\Cref{alg:main,alg:type_1_primes} referenced respectively in \Cref{thm:main-brief,thm:type_1_brief} are of a similar nature, and are implemented in the computer algebra system Sage \cite{sagemath}. Dealing with isogeny primes of Momose Type 2, however, requires a different approach. Rather than identifying integers which such isogeny primes must divide, we instead, on the one hand, identify a certain necessary condition (\Cref{cond:CC}) which such isogeny primes must satisfy; and then, on the other hand, determine an upper bound - conditional on GRH - for such isogeny primes. This is the only result in our work which requires GRH. The situation is summarised as follows, the details being given in \Cref{sec:momose_type_2}.

\begin{theorem}\label{thm:type_2_brief}
Let $k$ be a number field of degree $d$ and discriminant $\Delta_k$, and let $E/k$ be an elliptic curve admitting a $k$-rational $p$-isogeny of Momose Type 2. Then we have the following.
\begin{enumerate}
    \item $E$ is not semistable.
    \item The pair $(k,p)$ satisfies the necessary condition in \Cref{cond:CC}.
    \item Assuming GRH, $p$ satisfies
    \begin{equation}\label{eqn:type_2_intro_bound}
        p \leq (8d\log(12p) + 16\log(\Delta_k) + 10d + 6)^4
    \end{equation}
    and hence there are only finitely many isogeny primes of Momose Type 2.
\end{enumerate}
\end{theorem}

Determining whether a pair $(k,p)$ satisfies the necessary condition of \Cref{cond:CC} is a matter of checking splitting conditions on primes in the imaginary quadratic field $\Q(\sqrt{-p})$, and is thus easily checked via Legendre symbol computations. However, the conditional bound implied by \Cref{eqn:type_2_intro_bound} becomes rather large, as illustrated in \Cref{tab:output_summary_type_2}, which shows the conditional bounds on Type 2 primes for the number fields of smallest absolute discriminant and class number one for each $2 \leq d \leq 10$ which do not contain the Hilbert class field of an imaginary quadratic field. Checking the necessary condition of \Cref{cond:CC} on all primes up to this bound is the content of \Cref{alg:type_2_primes}, and is implemented in PARI/GP \cite{PARI}.

\begin{table}[htp]
\begin{center}
\begin{tabular}{|c|c|c|c|}
\hline
$d$ & $\Delta_K$ & LMFDB Label & Type 2 bound\\
\hline
$2$ & $5$ & $\href{https://www.lmfdb.org/NumberField/2.2.5.1}{2.2.5.1}$ & $5.65 \times 10^{10}$\\
$3$ & $49$ & $\href{https://www.lmfdb.org/NumberField/3.3.49.1}{3.3.49.1}$ & $4.09 \times 10^{11}$\\
$4$ & $125$ & $\href{https://www.lmfdb.org/NumberField/4.0.125.1}{4.0.125.1}$ & $1.46 \times 10^{12}$\\
$5$ & $14641$ & $\href{https://www.lmfdb.org/NumberField/5.5.14641.1}{5.5.14641.1}$ & $4.75 \times 10^{12}$\\
$6$ & $300125$ & $\href{https://www.lmfdb.org/NumberField/6.6.300125.1}{6.6.300125.1}$ & $1.12 \times 10^{13}$\\
$7$ & $594823321$ & $\href{https://www.lmfdb.org/NumberField/7.7.594823321.1}{7.7.594823321.1}$ & $2.65 \times 10^{13}$\\
$8$ & $64000000$ & $\href{https://www.lmfdb.org/NumberField/8.0.64000000.2}{8.0.64000000.2}$ & $4.16 \times 10^{13}$\\
$9$ & $16983563041$ & $\href{https://www.lmfdb.org/NumberField/9.9.16983563041.1}{9.9.16983563041.1}$ & $7.60 \times 10^{13}$\\
$10$ & $572981288913$ & $\href{https://www.lmfdb.org/NumberField/10.10.572981288913.1}{10.10.572981288913.1}$ & $1.24 \times 10^{14}$\\
\hline
\end{tabular}
\vspace{0.3cm}
\caption{\label{tab:output_summary_type_2}The bound, conditional on GRH, on Type 2 primes for the smallest Galois number fields of class number one not containing the Hilbert Class field of an imaginary quadratic field for each degree between $2$ and $10$.}
\end{center}
\end{table}

Combining \Cref{alg:main,alg:type_1_primes,alg:type_2_primes} into one - \Cref{alg:combined} - we may summarise \Cref{thm:main-brief,thm:type_1_brief,thm:type_2_brief} as follows. By \emph{semistable isogeny primes for $k$} we mean the isogeny primes for $k$ obtained by varying over only semistable elliptic curves.

\begin{theorem}\label{thm:combined}
Let $k$ be a number field. Then \Cref{alg:combined} outputs a finite set of primes $S_k$ such that, if $p$ is an isogeny prime for $k$ whose associated isogeny character is not of Momose Type 3, then, conditional on GRH, $p \in S_k$. In particular,
\begin{enumerate}
    \item if $k$ does not contain the Hilbert class field of an imaginary quadratic field, then $S_k$ contains $\IsogPrimeDeg(k)$;
    \item the above results are unconditional for the restricted set of semistable isogeny primes for $k$.
\end{enumerate}
\end{theorem}

Here we have highlighted that our results are unconditional for semistable isogeny primes since the restricted class of semistable elliptic curves arises naturally in the context of Frey curves and solving Diophantine equations \cite{freitas2015criteria}. Moreover, \Cref{alg:combined} does more than merely combine the different subalgorithms dealing with each of the isogeny types in Momose's classification: it also includes several methods for ruling out possible isogeny primes, based on congruence conditions, class field theory, and explicit computations with Jacobians of modular curves. These methods are discussed in \Cref{sec:automatic_weeding}.

To give the reader a sense of the size of this superset $S_k$ for $\IsogPrimeDeg(k)$ as well as the run-time, we show in \Cref{tab:output_summary} the output of the algorithm on the number fields from \Cref{tab:output_summary_type_2} for $2 \leq d \leq 6$, as well as the time taken for this to complete on an old laptop. Since the output necessarily contains $\IsogPrimeDeg(\Q)$ in each case, we show only those possible primes \emph{not} in $\IsogPrimeDeg(\Q)$. In addition, in \Cref{tab:output_summary} the timings refer to only checking possible Momose Type 2 primes up to $10^6$; checking up to the bounds given in \Cref{tab:output_summary_type_2} is the main bottleneck of the algorithm, and doing so for each number field in the table takes several hours of parallel computation in PARI/GP, which results in no additional possible isogeny primes beyond those listed.

\begin{table}[htp]
\begin{center}
\begin{tabular}{|c|c|c|c|c|}
\hline
$d$ & $\Delta_k$ & LMFDB Label & Possible Isogeny Primes & Time (s)\\
\hline
$2$ & $5$ & \href{https://www.lmfdb.org/NumberField/2.2.5.1}{2.2.5.1} & 23, 47 & $0.92$\\
$3$ & $49$ & \href{https://www.lmfdb.org/NumberField/3.3.49.1}{3.3.49.1} & 23, 29, 31, 73 & $3.23$\\
$4$ & $125$ & \href{https://www.lmfdb.org/NumberField/4.0.125.1}{4.0.125.1} & 23, 29, 31, 41, 47, 53, 61, 73, 97, 103 & $3.77$\\
$5$ & $14641$ & \href{https://www.lmfdb.org/NumberField/5.5.14641.1}{5.5.14641.1} & 23, 29, 31, 41, 47, 59, 71, 73, 97 & $35.42$\\[1ex]
\multirow{3}{*}{$6$} & \multirow{3}{*}{$300125$} & \multirow{3}{*}{\href{https://www.lmfdb.org/NumberField/6.6.300125.1}{6.6.300125.1}} & \multirow{3}{*}{\makecell{23, 29, 31, 41, 47, 53, 59, 61, 71, 73, \\ 79, 83, 97, 103, 107, 109, 113, 127, 131, 191,\\ 211, 263, 311, 503}} & \multirow{3}{*}{$206.62$}\\
& & & &\\
& & & &\\
\hline
\end{tabular}
\vspace{0.3cm}
\caption{\label{tab:output_summary}The possible isogeny primes $p \notin \IsogPrimeDeg(\Q)$ for the number fields from \Cref{tab:output_summary_type_2} for $2 \leq d \leq 6$, as well as the time taken to obtain this running the algorithm on an old laptop. The time only measures checking possible Momose Type 2 primes up to $10^6$; checking all up to the conditional bound adds several hours to the runtime in each case.}
\end{center}
\end{table}

Determining which of the primes $p$ in these supersets are actually isogeny primes for $k$ requires one to determine whether $X_0(p)(k)$ contains any noncuspidal points, a subject which has hardly been explicitly studied for $\deg(k) \geq 3$. However, by building on recent work of Box, Gajovi\'c and Goodman \cite{box2021cubic} in which the authors determine the finitely many cubic points on some modular curves $X_0(N)$, we are able to exactly determine $\IsogPrimeDeg(k)$ for some cubic fields, yielding the first instances of the determination of $\IsogPrimeDeg(k)$ for number fields of degree at least three.

\begin{theorem}\label{thm:first_cubic}
Assuming GRH, we have the following:
\begin{align*}
\IsogPrimeDeg(\Q(\zeta_7)^+) &= \IsogPrimeDeg(\Q),\\
\IsogPrimeDeg(\Q(\alpha)) &= \IsogPrimeDeg(\Q) \cup \left\{29\right\},\\
\IsogPrimeDeg(\Q(\beta)) &= \IsogPrimeDeg(\Q),
\end{align*}
where $\alpha^3 - \alpha^2 - 2\alpha - 20 = 0$ and $\beta^3 - \beta^2 - 3\beta + 1 = 0$.
\end{theorem}

In addition to cubic fields, our algorithm provides for several key improvements which allow us to determine several more instances of quadratic isogeny primes. The following is obtained by combining our algorithm with standard techniques for working explicitly with modular curves in Magma \cite{magma} (see Section 7 of \cite{banwait2021explicit} for an overview of these techniques).

\begin{theorem}\label{thm:more_quadratic}
Let $D \neq 1$ be a squarefree integer such that $|D| < 50$ and $\Q(\sqrt{D})$ is not imaginary quadratic of class number one. Then $\IsogPrimeDeg(\Q(\sqrt{D})) = \IsogPrimeDeg(\Q)$ unless $D$ is listed in \Cref{tab:quadratic_isogenies}, in which case the additional isogeny primes are listed in the column `New Isogeny Primes', and any primes not yet determined are listed in the column `Undetermined'.
\end{theorem}

\begin{table}[htp]
\begin{center}
\begin{tabular}{|c|c|c|}
\hline
$D$ & New Isogeny Primes & Undetermined\\
\hline
$-47$ & $31$ & 61\\
$-39$ & -- & 97\\
$-37$ & -- & 59, 131\\
$-31$ & 73 & --\\
$-23$ & 29, 31 & --\\
$-15$ & 23 & --\\
$-5$ & 23 & --\\
$5$ & 23, 47 & --\\
$13$ & 31 & --\\
$17$ & -- & 23\\
$29$ & 29 & --\\
$37$ & -- & 23\\
$41$ & 41 & --\\
$47$ & -- & 59\\
\hline
\end{tabular}
\vspace{0.3cm}
\caption{\label{tab:quadratic_isogenies}Determination of $\IsogPrimeDeg(\Q(\sqrt{D}))$ for squarefree $|D| < 50$, excluding the nine imaginary quadratic fields of class number one. If $D$ is not listed here, then $\IsogPrimeDeg(\Q(\sqrt{D})) = \IsogPrimeDeg(\Q)$. The primes in `New Isogeny Primes' have been verified to be isogeny primes; it is currently not known whether the primes in `Undetermined' are isogeny primes for $\Q(\sqrt{D})$ or not.}
\end{center}
\end{table}

The implementation of the combined \Cref{alg:combined} is available at:

\begin{center}
\url{https://github.com/isogeny-primes/isogeny-primes}
\end{center}

As with its predecessor \cite{quadratic_isogeny_primes}, it has been released as a command line tool under the GPLv3+ licence, and the \path{README.md} contains detailed instructions on its use and an overview of the testing strategy. All filenames will refer to files in this repository \cite{isogeny_primes}.

The outline of the paper is as follows. \Cref{sec:notation_and_prelims} sets the notation to be used throughout the paper as well as some basic results about roots of characteristic polynomials of Frobenius of elliptic curves over finite fields. The heart of the paper is \Cref{sec:strengthen_momose} which strengthens and fixes Momose's isogeny classification theorem, and presents the algorithmic version of it. \Cref{sec:generic} discusses some optimisation aspects of the implementation of the algorithm dealing with the so-called \emph{generic isogeny primes}, those which arise from isogeny characters which are not of Type\footnote{here we mean `signature Type' rather than `Momose Type'. This distinction will be made clear in \Cref{sec:strengthen_momose}.} $1$, $2$ or $3$. \Cref{sec:momose_type_1,sec:momose_type_2} deal respectively with isogeny primes arising from isogeny characters of Momose Type $1$ and $2$, and \Cref{sec:automatic_weeding} presents methods to further eliminate possible isogeny primes, based largely on methods of explicit class field theory. Combining all of the algorithms up to this point into the main \Cref{alg:combined} is done in \Cref{sec:combined}. Finally in \Cref{sec:cubic_examples} we present results related to cubic points on modular curves, and prove \Cref{thm:first_cubic}. (The verification required for \Cref{thm:more_quadratic} is given in \path{magma_scripts/QuadraticVerifs.m} in the above repository.)

\ack{
We are grateful to Edgar Costa, David Roe and Andrew Sutherland for providing us access to computational resources of the Simons
Collaboration in Arithmetic Geometry, Number Theory, and Computation, with which many of the computations have been performed. The copy of Magma was made available by a generous initiative of the Simons Foundation. We would also like to thank Josha Box for useful conversations and providing some Magma code, as well as Nicolas Billerey, Pip Goodman, Filip Najman and Andrew Sutherland for comments and corrections to an earlier version of the manuscript. The first named author is grateful to Jennifer Balakrishnan and Boston University for hosting a research visit where this version of the manuscript was finalised. We are also extremely grateful to the anonymous Referee for their careful and meticulous review of an earlier version of the manuscript.

During the final stages of this project we were deeply saddened to learn of the untimely passing of Bas Edixhoven, who was the second named author's PhD advisor. We were both touched by his support and care for more junior mathematicians, were inspired by his generous sharing of his time and ideas, and we wish to dedicate this paper in his memory.
}

\section{Notation and Preliminaries}\label{sec:notation_and_prelims}

In this section we set the notation for relevant objects to be used throughout the paper. Certain sections below will introduce their own notation in addition to those set here. This is the default notation to be taken, unless a particular result overrides the notation set here. We will also prove some results of a preliminary nature to be used later.

\begin{align*}
    d &: \mbox{an integer $\geq 1$}\\
    k &: \mbox{a number field of degree $d$}\\
    \Cl_k &: \mbox{the class group of $k$}\\
    h_k &: \mbox{the class number of $k$}\\
    G_k &: \Gal({\overline{k}/k}), \mbox{the absolute Galois group of $k$}\\
    K &: \mbox{the Galois closure of $k$ over $\Q$}\\
    \Sigma &: \Hom(k,K), \mbox{the embeddings of $k$ in $K$}\\
    p &: \mbox{a rational prime (denoting the isogeny prime we seek to bound)}\\
    \chi_p &: \mbox{the mod-$p$ cyclotomic character of $G_k$}\\
    \fp_0 &: \mbox{a chosen and fixed prime ideal of $K$ lying above $p$}\\
    q &: \mbox{a rational prime different from $p$}\\
    \fq &: \mbox{a prime ideal of $k$ lying above $q$}\\
    h_\fq &: \mbox{the order of $\fq$ in $\Cl_k$}\\
    \gamma_\fq &: \mbox{a generator of $\fq^{h_\fq}$ in $\Cl_k$}\\
    \F_\fq &: \mbox{the residue field of $\fq$}\\
    \frob_{\fq} &: \mbox{a choice of lift of Frobenius in $G_k$}\\
    \iota_\fq &: \mbox{the embedding of $k$ into its $\fq$-adic completion}\\
    E &: \mbox{an elliptic curve over $k$ admitting a $k$-rational $p$-isogeny}\\
    \lambda &: \mbox{the isogeny character of $E$; that is, the Galois action on the isogeny's kernel}\\
    \mu &: \lambda^{12}\\
    \eps &: \mbox{the signature of $\lambda$ (see \Cref{subsec:mom_lem_1})}\\
\end{align*}

Note that $\mu$ is unramified outside of $p$ (Propositions 1.4 and 1.5 in \cite{david2012caractere}; Propositions 3.3 and 3.5 in \cite{david2011borne}) and its image is abelian, hence $\mu(\frob_q)$ is well defined. Neither of these facts are true for $\lambda$ itself, which is why we take the $12$\textsuperscript{th} power; it is then expedient and minimises overload of notation to introduce a new symbol for $\lambda^{12}$.

\begin{remark}\label{rem:mu_galois_and_ideal}
    By precomposing with the Artin map from class field theory, we may view each of $\lambda$ and $\mu$ as character on the group of fractional ideals of $k$ coprime to $p$. This is done in a minority of places in the paper (specifically, \Cref{prop:momose_1_non_galois_md}, the proof of \Cref{lem:gal_act_eps}, and \Cref{ssec:ice}), and is done largely so that one may employ the following convenient shorthand notation for $\mu((\alpha))$:
$$\mu((\alpha)) = \prod_{\fq} \mu(\frob_\fq)^{v_\fq(\alpha)}.$$
\end{remark}

The following result will be used in \Cref{ssec:from_sig_types_to_momose}.

\begin{proposition}\label{prop:ideal_generators}
Let $L$ be an imaginary quadratic field, $q$ an odd rational prime, $\fq$ a prime of $\O_L$ above $q$, and $f$ a positive integer such that $\fq^f = \overline{\fq^f}$. Suppose that $\fq^f$ is principal, generated by $\alpha_\fq \in L$. Then $\alpha_\fq^{12}=\Nm_{L/\Q}(\fq)^{6f}$.
\end{proposition}
\begin{proof}
We split up the prove according to the splitting behavior of $q$ in $L$. Firstly, if $q$ splits in $L$ then $\fq^f \neq \overline{\fq^f}$ and there is nothing to prove. Secondly, if $q$ is inert in $L$ then $\fq=q\O_L$ and $\Nm(\fq) = q^2$. This implies $\alpha_\fq\O_L= q^f \O_L$ and hence that $\alpha_\fq/q^f$ is a unit. Since $L$ is imaginary quadratic this means $\alpha_\fq^{12} = q^{12f} = \Nm(\fq)^{6f}$.

Finally, suppose that $q$ ramifies in $L$. Then $\fq^2=q\O_L$ and $\Nm(\fq)=q$. This implies $\alpha_\fq^2\O_L= q^f \O_L$ and hence that $\alpha_\fq^2/q^f$ is a unit. Since $q>2$ is ramified in $L$ we know $L\neq \Q(i)$ and hence the order of this unit divides $6$. This means $\alpha_\fq^{12} = q^{6f} = \Nm(\fq)^{6f}$.
\end{proof}
\begin{remark}The condition $q>2$ is necessary, since $\fq := (i+1)\Z[i]$ is an ideal invariant under complex conjugation and $(i+1)^{12}=-64=-\Nm(\fq)^6 \neq \Nm(\fq)^6$. However when $q=2$ one can in general still conclude that $\alpha_\fq^{24}=\Nm(\fq)^{12f}$.
\end{remark}
\begin{remark}
The above proposition is in fact true if one replaces $\fq^f$ by an arbitrary principal ideal $I$ co-prime to 2 such that $I = \overline I$, but we won't need this.
\end{remark}

The remainder of this section will collect consequences and explicit statements derived from Theorem 4.1 of \cite{waterhouse1969abelian}, which is an explicit version of the Honda-Tate theorem that describes isogeny classes of abelian varieties in terms of characteristic polynomials of Frobenius. We will first state the theorem using notation which is compatible with the variables used in our work.

\begin{theorem}[Waterhouse]\label{thm:waterhouse}
Let $q$ be a prime and $f$ an integer. The polynomials that occur as the characteristic polynomial of Frobenius of an elliptic curve over $\F_{q^f}$ are exactly the polynomials of the form $x^2+tx+q^f$ with $t$ an integer such that $|t| \leq 2\sqrt{q^f}$ that satisfy one of the following conditions:
\begin{itemize}
    \item[1.] $(t,q)=1$; \label{item:waterhouse_1}
    \item[2.] $f$ is even and $t=\pm 2 q^{f/2}$;\label{item:waterhouse_2}
    \item[3.] $f$ is even, $q\nequiv 1 \Mod{3}$ and $t=\pm q^{f/2}$;\label{item:waterhouse_3}
    \item[4.] $f$ is odd, $q=2,3$ and $t=\pm q^{(f+1)/2}$;\label{item:waterhouse_4}
    \item[5.] \label{item:waterhouse_5}\begin{itemize}
        \item[i.]  $f$ is odd and $t=0$;
        \item[ii.] $f$ is even, $p\nequiv 1 \Mod{4}$ and $t=0$.
    \end{itemize}
\end{itemize}
The first case only occurs for ordinary elliptic curves, while the  other cases only occur for supersingular elliptic curves.
\end{theorem}

The description of the roots of the characteristic polynomials of Frobenius for supersingular elliptic curves given in \Cref{thm:waterhouse} is summarised in \Cref{tab:ss_roots}. In the sequel we refer to such roots as \textbf{supersingular Frobenius roots}, while by \textbf{ordinary Frobenius roots} we shall mean the roots of the characteristic polynomials of Frobenius for ordinary elliptic curves.

\begin{corollary}\label{cor:ord_frob_pow_not_rat}
Let $\beta$ be an ordinary Frobenius root of an elliptic curve over $\F_{q^f}$. Then for all $n \geq 1$, $\beta^n$ is not rational. 
\end{corollary}

\begin{proof}
Let $E/\F_{q^f}$ be an ordinary elliptic curve with characteristic polynomial of Frobenius equal to $(x - \beta)(x - \bar{\beta})$. Then $(x - \beta^n)(x - \bar{\beta}^n)$ is the characteristic polynomial of Frobenius of $E$ viewed over $\F_{q^{fn}}$. Since being ordinary is invariant under base change one has that $\beta^n$ is a root of a polynomial $x^2+t'x+q^{fn}$ of the form of  Case 1 of \Cref{thm:waterhouse}; that is, $\gcd(t',q) = 1$, and therefore $|t'| < 2\sqrt{q^f}$, which implies that the discriminant of this polynomial is negative. In particular $\beta^n$ is the root of an irreducible polynomial, and hence not rational.
\end{proof}

\begin{table}[]
$$
\renewcommand{\arraystretch}{1.5}
\begin{array}{|c|c|c|c|c|c|c|}
\hline
\text{Case of \Cref{thm:waterhouse}} & f                    & t                                & \beta                                                 & \beta^{12}               \\ \specialrule{.12em}{0em}{0em}  
2.                   & even                 & \pm 2 q^{f/2}                    & \mp q^{f/2}                                           & q^{6f}                   \\ \hline
3.                   & even                 & \pm q^{f/2}                      & \pm \zeta_3 q^{f/2} \text{ or } \pm \zeta_3^2 q^{f/2} & q^{6f}                   \\ \hline
\multirow{4}{*}{4.}  & \multirow{4}{*}{$odd$} & \multirow{2}{*}{$\pm 2^{(f+1)/2}$} & \pm(i-1)2^{(f-1)/2} \text{ or }                       & \multirow{2}{*}{$-q^{6f}$} \\ 
                    &                      &                                  & \pm(-i-1)2^{(f-1)/2}                                  &                          \\ \cline{3-5} 
                    &                      & \multirow{2}{*}{$\pm 3^{(f+1)/2}$} & \pm(\zeta_3-1)3^{(f-1)/2} \text{ or }                 & \multirow{2}{*}{$q^{6f}$}  \\ 
                    &                      &                                  & \pm(\zeta_3^2-1)3^{(f-1)/2}                           &                          \\ \hline
5.i                 & odd                  & \multirow{2}{*}{0}                               & \sqrt{-q^f}                                           & q^{6f}                   \\ \cline{1-2} \cline{4-5}
5.ii                & even                 &                                  & iq^{f/2}                                              & q^{6f}                   \\ \hline

\end{array}
$$
\caption{\label{tab:ss_roots}Description of the roots of the characteristic polynomials of Frobenius of supersingular elliptic curves.}
\end{table}

\begin{corollary}\label{cor:ss_frob_pow_rat}
Let $\beta$ be a supersingular Frobenius root of an elliptic curve over $\F_{q^f}$. Then $\beta^{12}$ is rational. 
\end{corollary}

\begin{proof}
This is clear from the last column of \Cref{tab:ss_roots}.
\end{proof}

\begin{corollary}\label{cor:ss_frob_pow_norm}
Let $E/k$ be an elliptic curve over a number field, and $\fq$ a prime of $k$ of odd residue characteristic $q$ such that $E$ admits potentially good supersingular reduction at $\fq$. Let $\beta$ be a root of the characteristic polynomial of Frobenius at $\fq$ acting on the $p$-adic Tate module of $E$ (for any $p \neq q$; this is independent of the choice of $p$). Then $\beta^{12} = \Nm(\fq)^6$.
\end{corollary}

\begin{proof}
By considering the different types of additive reduction from the Kodaira-N\'{e}ron classification of the special fibre of the N\'{e}ron model of $E$ over $\O_k$, the elliptic curve $E$ attains good reduction over a totally ramified extension $L$ of $k$ of degree $1$, $2$, $3$, $4$ or $6$. Since $L$ is totally ramified, the reduction of $E/L$ at the unique prime of $L$ above $\fq$ is an elliptic curve over the residue field $\F_\fq$ of $k$ at $\fq$, and hence $\beta$ is a root of the characteristic polynomial of Frobenius of a supersingular elliptic curve over $\F_\fq$. This elliptic curve over $\F_\fq$ is not unique, since it depends on the choice of $L$; however another choice of $L$ will only lead to a curve which is a quadratic, quartic or sextic twist of E.  Thus $\beta^{12}$ is independent of this choice.

Writing $f$ for the residue field degree of $\fq$, we see from the last column of \Cref{tab:ss_roots} that $\beta^{12} = \Nm(\fq)^6$. 
\end{proof}

We conclude this background section with a recap on Serre's fundamental characters.

\subsection{Fundamental characters of level $n$}\label{ssec:fundamental_chars}

The canonical reference for the material here is \cite[\textsection 1]{serre_prop_galois}.

For our number field $k$ and prime $p$, let $\fp$ be a prime of $k$ over $p$, and consider the tame inertia group $I_\fp = \Gal(\overline{k_\fp}/k^{nr}_\fp)$, where $k^{nr}_\fp$ is the maximal unramified extension of $k_\fp$ in $\overline{k_\fp}$. Both of these fields have the same residue field, which is an algebraic closure of $\F_p$ that we denote by $\overline{\F_p}$.

Inside $\overline{\F_p}$ one has all finite extensions $\F_{p^n}$ of $\F_p$ with norm maps $\F^\times_{p^m} \to \F^\times_{p^n}$ whenever $n | m$, and there is a natural identification $\theta$ of $I_\fp$ with the inverse limit of this system of norm maps \cite[Proposition 2]{serre_prop_galois}.

We therefore obtain, for each power $q = p^n$, the natural projection \[ \theta_{q-1} : I_\fp \to \F_q^\times; \] this may be considered \emph{the} fundamental character of level $n$, and by composing it with an automorphism of $\F_q^\times$ (i.e. a power $\phi^i$ of the Frobenius automorphism $\phi : x \mapsto x^p$) we obtain all other fundamental characters of level $n$, viz. $\theta_{q-1}^{p^i} := \phi^i \circ \theta_{q-1}$, for $i = 0, \ldots, n-1$.

\section{On Momose's Isogeny Classification Theorem}\label{sec:strengthen_momose}

In this section we generalise and make Theorem 1 of \cite{momose1995isogenies} algorithmic, and address the gaps and mistakes in the original proof in \emph{loc. cit.} mentioned in the Introduction.

\subsection{Strengthening Momose's Lemma 1}\label{subsec:mom_lem_1}

In this subsection we provide a version of Momose's Lemma 1 without the assumptions that $p$ is unramified in $k$ or that $k$ is Galois over $\Q$.

The first step is to describe the possible integers $a_\sigma$ which can occur in the statement of Momose's Lemma 1. A fuller description of these integers was given by David as Proposition 3.2 in \cite{david2011borne} under the assumption that $p$ is unramified in $k$. We thus begin by providing the more general version of this result of David.

Let $\fp$ be a prime of $k$ lying above $p$, $I_\fp  =\Gal(\overline{k_\fp}/k^{nr}_\fp)\subseteq \Gal(\overline{k_\fp}/k_\fp)$ the inertia subgroup at $\fp$ and let $\mu$ be the 12-th power of a $p$-isogeny character over $k_\fp$. Then we would like to describe the possible actions of $\mu$ restricted to $I_\fp$. 

This is provided by following result, which is a generalisation of Proposition 3.2 of \cite{david2011borne}, since for an unramified extension, the zeroeth fundamental character $\theta_{p-1}$ is equal to the cyclotomic character $\chi_p$ restricted to $I_\fp$. See \Cref{ssec:fundamental_chars} for more on fundamental characters.

\begin{proposition}\label{prop:local_momose_1}
Let $\mu=\lambda^{12}$ be the $12$-th power of a $p$-isogeny character, and $\fp | p$ a prime of $k$ with ramification index $e_\fp$. Suppose that $p \geq 5$. Then there exists an integer $0 \leq a_\fp \leq 12e_\fp$ such that $\mu|_{I_\fp} = \theta_{p-1}^{a_\fp}$. We have $a_{\fp} \equiv 0,4,6  \textrm{ or } 8 \Mod{12}$, and if $E$ is semistable, then $a_{\fp} \equiv 0 \Mod{12}$.
\end{proposition}

\begin{proof}
This proof is essentially the same as the proof of Proposition 3.2 of \cite{david2011borne}, the main difference being that we here are keeping track of the ramification index $e_\fp$ (which for David was equal to $1$). In the interest of being self-contained, we supply the details.

In the case that we have potentially multiplicative reduction we have $\lambda^{2}|_{I_{\fp}}$ is either trivial or $\chi_p^2$. Now since $\theta_{p-1}^{e_\fp} = \chi_p$ it follows that $\mu|_{I_\fp} = \theta_{p-1}^{a_\fp}$ where $a_\fp = 0$ or  $a_\fp = 12e_\fp$.

In the case that we have potentially good reduction there exists an integer $a'_\fp$ such that $\lambda|_{I_\fp} = \theta_{p-1}^{a'_\fp}$. Additionally there exists a purely ramified extension $k_\fp \subseteq k'_\fp$ over which the elliptic curve $E$ obtains good reduction. Writing $\widetilde{E_\fp}$ for this reduced elliptic curve, we have that the degree $e'_\fp := [k'_\fp : k_\fp]$ of this extension divides the size of the geometric automorphism group $\Aut(\widetilde{E_\fp})$ (see the proof of Theorem 2 in Section 2 of \cite{serretate}). Since $p \geq 5$, we obtain that $e'_\fp$ divides $2, 4$ or $6$. Let $I'_\fp$ denote the inertia subgroup of $\Gal(\overline{k'_\fp}/k'_\fp)$. Then $e'_\fp e_\fp$ is the ramification degree of $k'_\fp/\Q_p$. Writing $\theta'_{p-1}$ for the fundamental character of level one of $k'_\fp$, we have on $I'_\fp$ the equality $\theta_{p-1} = (\theta'_{p-1})^{e'_\fp}$. Since $E$ has good reduction over $k'_\fp$ we know by \cite[\textsection 1.13]{serre_prop_galois} that there exists an integer $0 \leq r_{\fp} \leq e'_\fp e_\fp$ such that on $I'_\fp$ we have $\lambda = (\theta'_{p-1})^{r_\fp}$. Putting all of these relations between characters on $I'_\fp$ together we obtain
$$(\theta'_{p-1})^{r_\fp} = (\theta_{p-1})^{a'_\fp} = (\theta'_{p-1})^{e'_\fp a'_\fp}.$$

Since the order of $\theta'_{p-1}$ is $p-1$ one obtains:
$$r_\fp \equiv e'_\fp a'_\fp \mod p-1.$$ Multiplying by the integer $12/e'_{\fp}$ gives $\frac {12} {e'_\fp} r_\fp \equiv \frac {12} {e'_\fp} e'_\fp a'_\fp \equiv 12 a'_\fp \mod p-1$, so that:
$$ \mu|_{I_\fp} = \lambda^{12}|_{I_\fp} = \theta_{p-1}^{a_\fp}$$ with $a_{\fp} = \frac {12} {e'_\fp} r_\fp$. Observe that the possible range of $r_\fp$ gives that $0 \leq a_\fp \leq 12e_\fp$.

We now show that $a_\fp = \frac {12} {e'_\fp} r_\fp \equiv 0, 4, 6 \textrm{ or } 8 \mod 12.$ This can be done by studying the possible values $1,2,3,4,6$ of $e'_\fp$ separately. In the cases $e'_\fp=1, 2$ or $3$ one has $\frac {12} {e'_\fp} \equiv 0, 4, 6 \textrm{ or } 8 \mod 12$, hence $a_\fp$ satisfies that congruence as well. If $e'_\fp=4$ respectively $6$, then $\frac {12} {e'_\fp} = 3$ respectively $2$. Furthermore, $r_p \equiv e'_\fp a'_\fp \mod p-1$ implies $r_\fp$ is even, which gives the result in these two cases also.

Finally we establish that, if $E$ is semistable, then we have $a_\fp \equiv 0 \Mod{12}$. If $E$ has multiplicative reduction, then this was established at the very beginning of the proof. If $E$ has good reduction, then this follows from observing that the integer $e'_\fp$ is equal to $1$, whence $a_\fp = 12r_\fp$.
\end{proof}

\begin{remark}
The character $\theta_{p-1}$ is surjective, so in particular the residue class $a_\fp \mod p-1$ is determined by $\mu$. Thus if $p-1 > 12e_\fp$ then the integer $a_\fp$ is unique.
\end{remark}

In order to get the results we need without the Galois assumption we will first make the following definition following Freitas and Siksek in \cite{freitas2015criteria} (see the discussion just before Proposition 2.2 in \emph{loc. cit.}).

\begin{definition}\label{def:signature}
Let $k$ be a number field with Galois closure $K$ over $\Q$. For each $\sigma \in \Sigma := \Hom(k,K)$ let $a_\sigma \in \left\{0,4,6,8,12\right\}$ be an integer and denote by $\eps = \sum_{\sigma}a_\sigma\sigma$ a formal sum. Then $\eps$ is called a {\bf $k$-isogeny signature}.
\end{definition}

For $\alpha \in k^\times$, and $\eps = \sum_{\sigma}a_\sigma\sigma$ a $k$-isogeny signature we prefer to use Momose's notation $\alpha^\eps$ to denote what David calls $\mathcal{N}(\alpha)$:
\begin{align*}
    (-)^\eps : k &\longrightarrow K\\
    \alpha &\longmapsto \prod_{\sigma \in \Sigma}\sigma(\alpha)^{a_\sigma}.
\end{align*}

We prefer this approach since it makes explicit the dependence on $\eps$.

\Cref{prop:local_momose_1} describes $\mu$ on the local inertia groups at the different primes $\fp$ above $p$ in terms of fundamental characters. However since $\mu$ is unramfied outside of $p$ these are the only inertia subgroups on which $\mu$ acts nontrivially.

Now let $k \subset k^\mu$ be the smallest abelian extension that trivializes $\mu$. Then $\mu$ induces an injective morphism $\overline \mu : \Gal(k^\mu/k) \to \F_p^\times$. Let $\I_k$ denote the idèles of $k$. Then class field theory provides the global reciprocity map $$r : \I_k \to \Gal(k^\mu/k).$$ The units $k^\times$ embed diagonally into $\I_k$ and one has $r(k^\times) = \set{\id_{k^\mu}}$, and composing the inclusions $k_\fp^\times \hookrightarrow \I_k$ with $r$ gives the local reciprocity map $r_\fp$ at $\fp$. Since the order of $\Gal(k^\mu/k)$ is coprime to $p$ the local reciprocity map vanishes on $1+\fp\O_{k_\fp}$ and hence we also get a map $\overline r_\fp:\F_\fp^\times \to \Gal(k^\mu/k)$. Let $\F_\fp$ denote the residue field at $\fp$ and write $p^n = \# \F_\fp$. Then the fundamental character of level $n$ can be seen as a map $\theta_{p^n-1}: I_\fp \to \F_\fp^\times.$ Finally by choosing an embedding $k^\mu \to \overline k_\fp$ we get a map $\phi_\fp:I_\fp \subseteq  \Gal(\overline{k_\fp}/k_\fp) \to \Gal(k^\mu/k)$. This can be summarized in the following commutative diagram:

\centerline{\xymatrix{
 \O_{k_\fp}^\times \ar@{->>}[rr] \ar@{^(->}[d]&& \F_\fp^\times \ar[d]_{\alpha \mapsto \overline r_\fp(\alpha^{-1})} \ar@{-->}[drr]^(.3){??} && I_\fp \ar@{->>}[ll]_{\theta_{p^n-1}} \ar[d]^{\mu|_{I_\fp}=\theta_{p-1}^{a_\fp}}\ar[dll]_(.3){\phi_\fp}\\
  k_\fp^\times \ar[rr]_{\alpha \mapsto r_\fp(\alpha^{-1})\quad} && \Gal(k^\mu/k) \ar[rr]_{\quad\quad \overline \mu} && \F_p^\times
}}
where the equality $\phi_\fp(s) = \overline r_\fp \circ \theta_{p^n-1}(s^{-1})$ is Proposition 3 of \cite{serre_prop_galois}. Since $\theta_{p-1} = \Nm_{\F_\fp/\F_p} \circ \theta_{p^n-1}$ we know that the map that is needed at the place of the question marks to make everything commute is $x \mapsto \Nm_{\F_\fp/\F_p}(x)^{a_\fp}.$

In particular this commutative diagram allows one to describe $\overline \mu \circ r_\fp : \O_{k_\fp}^\times \to \F_p^\times$ more directly by: 
    $$\overline \mu \circ r_\fp(\alpha) = \Nm_{\F_\fp/\F_p}(\alpha \Mod{\fp})^{-a_\fp}. $$

Having this concise description of $\overline \mu \circ r_\fp$ on $\O_{k_\fp}^\times$ for all primes $\fp | p$ using local class field theory, as well as the fact that $\mu$ is unramified outside $p$, allows one to describe $\mu$ more generally as in the following Proposition, in which we have used the ideal theoretic description of class field theory to see $\mu$ as a character on the fractional ideals coprime to $p$ (c.f. \Cref{rem:mu_galois_and_ideal}).

\begin{proposition}[Generalisation of Proposition 2.6 in \cite{david2012caractere}]\label{prop:momose_1_non_galois_md}
Let $k$ be a number field, $K$ its Galois closure over $\Q$ and $\mu$ the 12-th power of a $p$-isogeny character over $k$. Then for every prime ideal $\fp_0$ lying above $p$ in $K$ there exists a  $k$-isogeny signature $\eps=\eps_{\fp_0}=\sum_{\sigma}a_\sigma\sigma$ such that for all $\alpha \in k^\times$ prime to $p$,
\[ \mu((\alpha)) \equiv \alpha^\varepsilon \Mod{\fp_0}.\]

Furthermore if $p > 13 $ and $p$ is unramified in $k$, then for every $\fp_0$ there is an unique such signature $\eps_{\fp_0}$.
\end{proposition}

\begin{proof}
View $\overline\mu \circ r$ as a character on $\I_k$ that vanishes on $k^\times$. Since $\mu$ is the $12$-th power of another character, $\overline\mu \circ r_v$ is trivial for all infinite places $v$ of $k$. We can use this to compute
\begin{align*}
    1   &= \prod_{\fq \nmid p} \overline \mu \circ r_\fq(\alpha_\fq) \times \prod_{\fp \mid p} \overline \mu \circ r_\fp(\alpha_\fp)\\
     &=\prod_{\fq \nmid p} \mu(\frob_\fq)^{v_\fq(\alpha)} \times \left(\prod_{\fp \mid p} \Nm_{\F_\fp/\F_p}(\alpha \mod \fp)^{a_\fp} \right)^{-1}
\end{align*}

If $p$ were unramified then we could write $$\Nm_{\F_\fp/\F_p}(\alpha \Mod{\fp}) \equiv \prod_{\sigma \in \Hom(k_\fp, \overline{\Q_\fp})} \sigma(\alpha) \Mod{\fp'},$$
where $\fp'$ is the prime of $\overline{\Q_\fp}$ lying over $p$, and the proposition would follow by setting $a_\sigma := a_{\sigma^{-1}(\fp_0)}$ and rewriting as follows:
\begin{align*}
\prod_{\fp \mid p} \Nm_{\F_\fp/\F_p}(\alpha \Mod{\fp})^{a_\fp} &\equiv \prod_{\fp \mid p}  \prod_{\sigma \in \Hom(k_\fp, \overline{\Q_\fp})} \sigma(\alpha)^{a_\fp} \Mod{\fp'}\\
&\equiv \prod_{\sigma \in \Hom(k,K)} \sigma(\alpha)^{\alpha_\sigma} \Mod{\fp_0} = \alpha^\eps \Mod{\fp_0},
\end{align*}
where in the last step we choose an embedding $K \subseteq K_{\fp_0} \hookrightarrow \overline{\Q_p}$.

Suppose now that a prime $\fp|p$ ramifies in $k$. Let $e_\fp$ denote the ramification index at $\fp$. Then we can write $$\Hom(k_\fp, \overline{\Q_p}) = \bigcup_{i=1}^{e_\fp} S_{\fp,i}$$ as the union of $e_\fp$ different sets of size $[\F_\fp:\F_p]$, so that for each $S_{\fp,i}$ we have:
$$\Nm_{\F_\fp/\F_p}(\alpha \Mod{\fp}) \equiv \prod_{\sigma \in S_{\fp,i}} \sigma(\alpha) \Mod{\fp'}.$$
Furthermore the conditions on $a_\fp$ from \Cref{prop:local_momose_1} allow us to write $a_\fp = \sum_{i=1}^{e_\fp} a_{\fp,i}$ with each $0 \leq a_{\fp,i} \leq 12$ and $a_{\fp,i} \equiv 0,4,6$ or $8 \mod 12$, whence
$$\Nm_{\F_\fp/\F_p}(\alpha \Mod{\fp})^{a_\fp} \equiv \prod_{i=1}^{e_\fp}\prod_{\sigma \in S_i} \sigma(\alpha)^{a_{\fp,i}} \Mod{\fp'}.$$

Writing 
\[\Hom(k,K) = \Hom(k,\overline{\Q_p}) =  \bigcup_{\fp|p}\bigcup_{i=1}^{e_\fp} S_{\fp,i}\]
and setting $a_\sigma = a_{\fp,i}$ for $\sigma \in S_{\fp,i}$ we can rewrite 
\begin{align*}
\prod_{\fp \mid p} \Nm_{\F_\fp/\F_p}(\alpha \Mod{\fp})^{a_\fp} &\equiv \prod_{\fp \mid p}  \prod_{i=1}^{e_\fp} \prod_{\sigma \in S_{\fp,i}} \sigma(\alpha)^{a_{\fp,i}} \Mod{\fp'}\\
&\equiv \prod_{\sigma \in \Hom(k,K)} \sigma(\alpha)^{\alpha_\sigma} \Mod{\fp_0} = \alpha^\eps \Mod{\fp_0}
\end{align*}
as in the unramified case and we are done.
\end{proof}

We refer to $\eps_{\fp_0}$ as \textbf{the isogeny signature of $\lambda$ with respect to $\fp_0$}. Note that since $\Gal(K/\Q)$ acts transitively on the primes above $p$, a different choice of prime $\fp_0$ merely permutes the integers $a_{\sigma}$. We will therefore often drop the subscript $\fp_0$ and speak of $\eps$ as \textbf{the isogeny signature of $\lambda$}. Note that fixing an ordering to the embeddings in $\Sigma = \Hom(k,K)$ allows one to think of the signature as a $d$-tuple valued in the set $\left\{0,4,6,8,12\right\}$, and hence one sees that there are precisely $5^d$ possible isogeny signatures for a degree $d$ number field. If all the integers $a_{\sigma}$ are the same integer $a$, then clearly the ordering on $\Sigma$ does not matter, and in the sequel we denote this signature as the $d$-tuple $(a,\ldots, a)$.

From the construction of $\eps_{\fp_0}$ in the above proof, we obtain the following condition which must be satisfied by the integers $a_\sigma$.

\begin{corollary}\label{cor:admissible_signatures}
Let $\eps_{\fp_0} = \sum_{\sigma \in \Sigma}a_\sigma\sigma$ be an isogeny signature with respect to $\fp_0$. Then, for $\sigma, \tau \in \Sigma = \Hom(k,K)$,
\[ \sigma^{-1}(\fp_0) = \tau^{-1}(\fp_0) \Longrightarrow a_\sigma \equiv a_\tau \Mod{p-1}.\]
In particular, if $p$ is inert in $k$ and $p \geq 17$, then all the integers $a_\sigma$ are the same.
\end{corollary}

Isogeny signatures satisfying this condition will be referred to as \textbf{admissible}. Note furthermore that if one of the integers in the signature $\varepsilon$ is $4$ or $8$ (respectively $6$), then from the proof of \Cref{prop:local_momose_1} we have that $\Aut(\widetilde{E}_\fp) \cong \mu_6$ and $p \equiv 2 \Mod{3}$ (respectively $\Aut(\widetilde{E}_\fp) \cong \mu_4$ and $p \equiv 3 \Mod{4}$). We therefore refer to such signatures as \textbf{sextic} (respectively \textbf{quartic}). The overlap of these two designations (i.e. the signature contains both $4$ and $6$ or both $8$ and $6$) is an interesting source of isogenies, and will be referred to as \textbf{mixed}. Signatures, all of whose defining integers are the same, will be referred to as \textbf{constant}.

\begin{example}

In the case that $X_0(p)$ is elliptic or hyperelliptic, one may generate quadratic points on $X_0(p)$ by taking the pullback of rational points along the hyperelliptic map to $\P^1$. The corresponding isogeny $\phi$ may be constructed in Magma. After passing to a suitable ramified extension of $k_\fp$ over which $E$ obtains semistable reduction, we let $\widehat{\phi}$ be the power series representation of $\phi$ on the formal groups of the elliptic curves. Then as in \cite[\S 1.10]{serre_prop_galois}, we may directly compute the $a_\fp$ integers as the $\fp$-adic valuation of the coefficient of $X$ in $\widehat{\phi}(X)$. In this way one may actually compute the signature of isogenies, as the following examples show; in both cases $\sigma$ denotes the nontrivial automorphism of the corresponding quadratic field.

\begin{enumerate}
    \item Let $k = \Q(\sqrt{-1120581})$, and let $E/k$ be an elliptic curve with $j$-invariant
    \[ \frac{1}{837568512}\left(-344992121\sqrt{-1120581} - 182301639894\right).\]
    Then $E$ admits a $k$-rational $11$-isogeny of signature $6\id + 8\sigma$, and in particular shows that isogenies of mixed signature exist.\\
    
    \item Let $k = \Q(\sqrt{38731793})$, and let $E/k$ be an elliptic curve with $j$-invariant
    \[ 4329499018988087705974500\sqrt{38731793} + 26944581751932950083389335625.\]
    Then $E$ admits a $k$-rational $23$-isogeny of signature $(6,6)$.
\end{enumerate}

These (and other) examples were found, and can be verified, with the Magma code in \path{magma_scripts/EpsilonTypes.m}; this script has also been used to generate test cases for our software package.
\end{example}

\begin{remark}
Note that Larson and Vaintrob have similarly removed the Galois assumption in Momose's Lemma 1, using the notion of \emph{algebraic characters}; see Definition 2.2 and Corollary 2.4 of \cite{larson_vaintrob_2014}. The left-hand side of the equation in their Corollary 2.4 is the space of algebraic characters, whilst the right-hand side may be identified with the space of isogeny signatures.
\end{remark}

\subsection{A general divisibility criterion}\label{subsec:divisibility}

Let $\fq$ be a prime ideal of $k$ which is coprime to $p$, and consider the reduction of $E$ at $\fq$, which is either potentially multiplicative, potentially good supersingular, or potentially good ordinary. In both of the potentially good cases, the characteristic polynomial of Frobenius $\Frob_\fq$ acting on the $p$-adic Tate module of $E$ has coefficients in $\Z$ and is independent of $p$ (Theorem 3 in \cite{serretate}); we may thus write $P_\fq(X)$ for this polynomial. This is a quadratic polynomial whose roots have absolute value $\sqrt{\Nm(\fq)}$. We write $L^\fq$ for the splitting field of this polynomial, which is either $\Q$ or an imaginary quadratic field.

In all cases of reduction at $\fq$, one obtains congruence conditions modulo $p$ on $\lambda(\Frob_\fq)$. If $E$ has potentially multiplicative reduction at $\fq$, then $\lambda(\Frob_\fq)$ is either $1$ or $\pm \Nm(\fq) \Mod{p}$ (Proposition 1.4 in \cite{david2012caractere}, Proposition 3.3 in \cite{david2011borne}). In the potentially good reduction case, for $\mathcal{P}^\fq$ a prime of $L^\fq$ above $p$, the images of the roots of $P_\fq(X)$ in $\mathcal{O}_{L^\fq}/\mathcal{P}^\fq$ are in $\F_p^\times$, and there is a root $\beta_\fq$ of $P_\fq(X)$ such that $\lambda(\Frob_\fq) = \beta_\fq \Mod{\mathcal{P}^\fq}$ (Proposition 1.8 in \cite{david2012caractere}, Proposition 3.6 in \cite{david2011borne}). For simplicity we will sometimes drop the $\fq$ subscript on $\beta_\fq$, particularly when `looping' over several such roots.

Write $h_\fq$ for the order of $\fq$ in the class group $\Cl_k$ of $k$, and let $\gamma_\fq$ be a generator of the principal ideal $\fq^{h_\fq}$. We apply \Cref{prop:momose_1_non_galois_md} to $\gamma_\fq$ and obtain, with the ideal-theoretic interpretation for the domain of $\mu$ as \Cref{rem:mu_galois_and_ideal}, the following expression:
\[ \mu((\gamma_\fq)) \equiv \gamma_\fq^\varepsilon \Mod{\fp_0}.\]
Replacing this with the Galois character interpretation for $\mu$, and using $\fq^{h_\fq} = (\gamma_\fq)$, we obtain:
\[ \mu^{h_\fq}(\Frob_\fq) \equiv \gamma_\fq^\varepsilon \Mod{\fp_0}.\]
This expression, when combined with the aforementioned congruence conditions modulo $p$ on $\mu(\Frob_\fq)$, yield divisibility conditions for $p$, namely that $p$ must divide one of the integers defined as follows.

\begin{definition}\label{def:ABC}
Let $k$ be a number field with Galois closure $K$ over $\Q$, $\fq$ a prime ideal of $k$ of order $h_\fq$ in the class group of $k$, and $\gamma_\fq$ a generator of the principal ideal $\fq^{h_\fq}$. For a set $S$ of integers, let $\lcm(S)$ denote the least common multiple of the integers in $S$, with the convention that this will be $0$ if $0 \in S$. Then we define the following integers:
\begin{align*}
A(\eps, \fq) &:= \Nm_{K/\Q}(\gamma_\fq^\eps - 1),\\
B(\eps, \fq) &:= \Nm_{K/\Q}(\gamma_\fq^\eps - \Nm(\fq)^{12h_\fq}),\\
C_s(\eps, \fq) &:= \lcm(\left\{ \Nm_{K(\beta)/\Q}(\gamma_\fq^\eps - \beta^{12h_\fq}) \ | \ \beta \mbox{ is a supersingular Frobenius root over }\F_\fq\right\}),\\
C_o(\eps, \fq) &:= \lcm(\left\{ \Nm_{K(\beta)/\Q}(\gamma_\fq^\eps - \beta^{12h_\fq}) \ | \ \beta \mbox{ is an ordinary Frobenius root over }\F_\fq\right\}),\\
C(\eps, \fq) &:= \lcm\left(C_o(\eps, \fq), C_s(\eps, \fq)\right),
\end{align*}
where in $C_s(\eps, \fq)$ (respectively $C_o(\eps, \fq)$)  the $\lcm$ is taken over all roots $\beta$ of characteristic polynomials of Frobenius of supersingular (respectively, ordinary) elliptic curves defined over the residue field $\F_\fq$ of $k$ at $\fq$.
\end{definition}

\begin{remark}\label{rem:ss_frob_roots_waterhouse}
Using Waterhouse's \Cref{thm:waterhouse}, specifically \Cref{tab:ss_roots} derived from it, one may show that
\[ C_s(\varepsilon, \fq) := \begin{cases*}
    B(2\varepsilon, \fq) & if $|\F_\fq| = 2^f$ with $f$ odd\\
    \Nm_{k/\Q}(\gamma_\fq^\varepsilon - \Nm(\fq)^{6h_\fq}) & otherwise.
    \end{cases*} \]
\end{remark}

We frame the above discussion as follows.

\begin{corollary}\label{cor:ABC_divs}
\begin{enumerate}
\item
If $E$ has potentially multiplicative reduction at $\fq$, then $p$ divides either $A(\varepsilon, \fq)$ or $B(\varepsilon, \fq)$.
\item
If $E$ has potentially good ordinary reduction at $\fq$, then $p$ divides $C_o(\varepsilon, \fq)$.
\item
If $E$ has potentially good supersingular reduction at $\fq$, then $p$ divides $C_s(\varepsilon, \fq)$.
\end{enumerate}
\end{corollary}


\subsection{Removing the Galois assumption in Momose's Lemma 2}

Having identified integers which $p$ must divide, the question of the non-zero-ness of these integers becomes relevant. This is the motivation for Momose's Lemma 2, or David's Proposition 2.15 in \cite{david2012caractere}, both of which assume that $k$ is Galois over $\Q$. 

We thus provide the following non-Galois version of Momose's Lemma 2. The proof is modelled on David's proof, and many of the arguments carry over \emph{mutatis mutandis}, though the details in Type 3 require some additional ideas. For $L$ a subfield of $k$, we denote by $\Sigma_L \subset \Sigma$ the subset of embeddings of $k$ in $K$ which act as the identity on $L$.

\begin{proposition}\label{prop:momose_lemma_2_non_galois}
Let $p \geq 17$, let $q \neq p$ be a rational prime which splits completely in $k$, and let $\fq$ be a prime of $k$ over $q$. If the condition shown in the left-most column of \Cref{tab:david_types} is satisfied, then the corresponding assertions in the rest of the table hold.
\end{proposition}

{
\small
\begin{table}[htp]
\begin{center}
\begin{tabular}{|c|c|c|c|c|c|c|}
\hline
Condition & $\gamma_\fq^\varepsilon$ & $\Q(\beta)$ & $\gamma_\fq^\varepsilon \in \Q$? & $(a_\tau)_{\tau \in \Sigma}$ & \thead{Everywhere \\ unramified\\ character} & \thead{Signature \\ Type}\\
\hline
$A(\varepsilon, \fq) = 0$ & $1$ & & \multirow{4}{*}{Yes} & All $0$ & $\mu$ & \multirow{2}{*}{Type 1}\\\cline{1-2} \cline{5-6}
$B(\varepsilon, \fq) = 0$ & $q^{12h_\fq}$ & & & All $12$ & $\mu/\chi_p^{12}$ & \\
\cline{1-3} \cline{5-7}
\multirow{2}{*}{$C_s(\varepsilon, \fq) = 0$} & \multirow{2}{*}{$\pm q^{6h_\fq}$} & \multirow{2}{*}{$\Q(\sqrt{-q})$} & & \multirow{2}{*}{All $6$} & \multirow{2}{*}{$\mu/\chi_p^{6}$} & \multirow{2}{*}{Type 2}\\
& & & & & &\\\hline
\multirow{6}{*}{$C_o(\varepsilon, \fq) = 0$} & \multirow{6}{*}{\makecell{$\beta^{12h_\fq}$ for $\beta$ \\ an ordinary \\ Frobenius root \\ over $\F_\fq$}} & \multirow{6}{*}{\makecell{$\Q(\beta) \subseteq k$, \\ $p$ splits or \\ ramifies in $\Q(\beta)$, \\ $\Nm_{k/\Q(\beta)}(\fq) = $ \\ $(\beta)$ or $(\overline{\beta})$}} & \multirow{6}{*}{No} & \multirow{3}{*}{\makecell{$a_\tau = 12$ for \\ $\tau \in \Sigma_{\Q(\beta)}$; \\ $0$ otherwise}} & & \multirow{6}{*}{Type 3}\\
 & & & & & &\\
& & & & & & \\\cline{5-5}
& & & & \multirow{3}{*}{\makecell{$a_\tau = 0$ for \\ $\tau \in \Sigma_{\Q(\beta)}$; \\ $12$ otherwise}} & & \\
& & & & & &\\
& & & & & &\\
\hline
\end{tabular}
\vspace{0.3cm}
\caption{\label{tab:david_types}Summary of what happens if one of the integers $A$, $B$, $C_s$ or $C_o$ is zero.}
\end{center}
\end{table}
}
\begin{proof}
First observe, from the assumption that $q$ splits completely in $k$, we have that $q$ splits completely in $K$; see for example Chapter 1, Section 9, Exercise 4 in \cite{neukirch}. This will be used throughout the proof.

Before considering the various cases, we set the following notation to be used throughout the proof. Let $\fq_1, \ldots, \fq_d$ be the distinct prime ideals of $k$ lying over $q$, and we suppose $\fq = \fq_1$. We label the embeddings in $\Sigma$ as $\sigma_1, \ldots, \sigma_d$, where $\sigma_1$ is the inclusion $k \subseteq K$. We let $M = \Gal(K/\sigma_1(k)) =\Gal(K/k)$, viewed as a subgroup of $G := \Gal(K/\Q)$; let $n$ denote the size of $M$.  For each $1 \leq i \leq d$, we let $Q_1^{(i)}, \ldots, Q_n^{(i)}$ be the distinct prime ideals of $K$ dividing $\sigma_i(\fq)\O_K$ (recalling that if $q$ splits completely in $k$ then $q$ splits completely in $K$). Because $q$ is unramified, it follows that $\sigma_i(\fq)\O_K = \prod_{j=1}^nQ_j^{(i)}$. Write $\eps = \sum_{i=1}^da_i\sigma_i$. We define the set
\[ S_q := \left\{ Q_j^{(i)} : 1 \leq i \leq d, 1 \leq j \leq n\right\}, \]
which is \emph{a priori} only a subset of the prime ideals of $K$ lying over $q$. However, since $S_q$ is stable under $G$, and $G$ acts transitively on the set of primes in $K$ lying over $q$, we have that $S_q$ actually equals the set of all primes lying over $q$. Finally we let $\tau_1, \ldots, \tau_d$ denote the (left) coset representatives of $G/M$, chosen such that $\tau_i \circ \sigma_1 = \sigma_i$. Fixing $\hat{Q} = Q_1^{(1)}$, this implies that
\[ \sigma_i(\fq)\O_K = \tau_i(\sigma_1(\fq))\O_K=\prod_{Q \in \tau_iM\hat{Q}} Q.\]

Since $G$ acts faithfully on the primes lying over $q$, this means that $\sigma_i(\fq)$ and $\sigma_j(\fq)$ are coprime if $i \neq j$. 


We now come to the three cases for the value of $\gamma_\fq^\eps$.

Consider first that $\gamma_\fq^\eps = 1$. By considering the ideal of $\O_K$ generated by $\gamma_\fq^\eps$, we obtain
\[ \O_K = \prod_{i=1}^d\sigma_i(\fq)^{a_ih_\fq}. \]
Since the $\sigma_i(\fq)$ are pairwise coprime we obtain that $a_i = 0$ for all $i$. Since $\mu|_{I_\fp} = \chi_p^{a_\fp}$ and $\mu$ is unramified away from $p$, we obtain that $\mu$ is unramified everywhere.

In the second case $\gamma_\fq^\eps = q^{12h_\fq}$, one similarly obtains that $a_i = 12$ for all $i$, whence $\mu/\chi_p^{12}$ is everywhere unramified.

In the third case, we have that $\gamma_\fq^\eps = \beta^{12h_\fq}$, where $\beta$ is a root of the characterisitic polynomial of Frobenius of an elliptic curve over $\F_\fq$. We write $L = \Q(\beta)$, which is either $\Q$ or an imaginary quadratic field. By assumption, the element $\gamma_\fq^\eps$ - \emph{a priori} in $K$ - is also in $L$; so either $\gamma_\fq^\eps$ is rational, or it generates $L$ and therefore $L$ is contained in $K$.

In the first of these two subcases, $\beta^{12h_\fq}$ is rational; therefore it is equal to its complex conjugate $\bar{\beta}^{12h_\fq}$; in particular there is a $12h_\fq$\textsuperscript{th} root of unity $\zeta \in L$ such that $\beta = \zeta \bar{\beta}$. However, since $L$ is imaginary quadratic, it only admits $n$\textsuperscript{th} roots of unity for $n = 2, 4$ or $6$; thus $\beta^{12}$ is rational. Moreover, since $\beta$ is an algebraic integer, $\beta^{12}$ is an integer. Since the absolute value of $\beta$ is $\sqrt{\Nm({\fq})} = \sqrt{q}$, we get that $\beta^{12} = \pm q^6$, and therefore
\[ \prod_{\sigma \in \Sigma}\sigma(\fq^{h_\fq})^{a_\sigma}\O_K= \gamma_\fq^\eps\O_K = \beta^{12h_\fq}\O_K = q^{6h_\fq}\O_K = \left(\prod_{\sigma \in \Sigma}\sigma(\fq)\O_K\right)^{6h_\fq}.\]
Since $\sigma_i(\fq)\O_K$ and $\sigma_j(\fq)\O_K$ are coprime for $i \neq j$ we obtain that all $a_\sigma$ are equal to $6$. That $\mu/\chi_p^{6}$ is everywhere unramified follows as previously. To show that $\Q(\beta) = \Q(\sqrt{-q})$ one observes from \Cref{cor:ord_frob_pow_not_rat,cor:ss_frob_pow_rat} that $\beta$ is a supersingular Frobenius root, so one simply checks through the possibilities in \Cref{tab:ss_roots} (observing that we have $f = 1$ because we are assuming $\fq$ is completely split).

In the second subcase, $L$ is an imaginary quadratic field contained in $K$. We let $H := \Gal(K/L)$. As in the first subcase, we consider the $K$-ideal
\[ (\gamma_\fq^\eps) = \prod_{i=1}^d\left(\prod_{j=1}^nQ_j^{(i)}\right)^{a_ih_\fq}; \]
the crucial observation is that, since this ideal is invariant under $H$ (because $\gamma_\fq^\eps \in L$), it must factor in the form
\begin{equation}\label{eqn:split_norm}
(\gamma_\fq^\eps) = \left (\prod_{Q \in H\hat{Q}} Q\right)^{a_{\id}}\left(\prod_{Q \in  H\gamma\hat{Q}} Q\right)^{a_\gamma}
\end{equation}
for $\gamma \in G$ an element which induces complex conjugation on $L$, and $a_{\id}$ and $a_{\gamma}$ are integers; in particular, there are only two choices (viz. $a_{\id}$ or $a_\gamma$) for each $a_i$.

Suppose for a contradiction that $M \not \subseteq H$. Choose $\delta \in M \backslash H$, and consider the two primes $\hat{Q}$ and $\delta\hat{Q}$ in $S_q$. These are both ideals dividing $(\gamma_\fq^\eps)$ with the same exponent (viz. $a_1h_\fq$), but are in different $H$-orbits under the $H$-action on $S_q$. This forces $a_{\id} = a_\gamma$, whence $\gamma_\fq^\eps = \Nm_{K/\Q}(\sigma_1(\gamma_\fq))^{a_\gamma}$; i.e., $\gamma_\fq^\eps \in \Q$, contradicting the subcase that we are currently in. Therefore $M\ \subseteq H$, which establishes that $L$ is contained in $k$.

We now have
\begin{align*}
    (\gamma_\fq^\eps)\O_L &= \left(\Nm_{k/L}(\gamma_\fq)\O_L\right)^{a_{\id}}\cdot\left(\overline{\Nm_{k/L}(\gamma_\fq)}\O_L\right)^{a_\gamma} \mbox{ (from \Cref{eqn:split_norm}) } \\
    \Rightarrow (\beta\O_L)^{12h_\fq} &= \left( (\Nm_{k/L}(\fq))^{a_{\id}}(\overline{\Nm_{k/L}(\fq)})^{a_\gamma}\right)^{h_\fq}.
\end{align*}

Since $q = \beta \bar{\beta}$, we obtain that the $L$-ideal $\Nm_{k/L}(\fq)$ is either $\beta\O_L$ or $\bar{\beta}\O_L$, which yields that the pair of integers $(a_{\id}, a_\gamma)$ is either $(0,12)$ or $(12,0)$, yielding the two possible forms of $\eps$ as in the table.

The only remaining assertion to prove is that $p$ splits or ramifies in $L = \Q(\beta)$ in the Type 3 case. Suppose for a contradiction that $p$ is inert in $L$, and let $\Frob_{\fp_0} \in \Gal(K/\Q)$ be a choice of a lift of Frobenius at $\fp_0$. The automorphism $\Frob_{\fp_0}$ satisfies the following two properties:
\begin{enumerate}
    \item it fixes $\fp_0$;
    \item its restriction to $L$ is also a lift of Frobenius of the prime ideal $p\mathcal{O}$, of residue class degree $2$, and hence its restriction to $L$ is the nontrivial element of $\Gal(L/\Q)$.
\end{enumerate}
We now consider the two embeddings of $k$ in $K$: $\sigma_1$ and $\Frob_{\fp_0} \circ \sigma_1$. By property (1) above, the preimage of $\fp_0$ in $k$ under each of these embeddings is the same, and therefore, by \Cref{cor:admissible_signatures}, $a_1 \equiv a_{\Frob_{\fp_0} \circ \sigma_1} \Mod{p-1}$. On the other hand, property (2) forces one of these integers to be $0$, and the other $12 \Mod{p-1}$. We therefore obtain a contradiction for primes $p$ such that $p - 1 \nmid 12$ (and hence for all $p \geq 17$).
\end{proof}

\Cref{prop:momose_lemma_2_non_galois} suggests that the signatures identified in \Cref{tab:david_types} require particular attention. This motivates the following definition.

\begin{definition}\label{def:sig_types}
Let $\eps = \sum_{\sigma}a_\sigma\sigma$ be an isogeny signature: 
\begin{itemize}
    \item  If $\eps = 0\sum_{\sigma}\sigma$ or $\eps = 12\sum_{\sigma}\sigma$, then $\eps$ is of \textbf{Type 1}.
    \item  If $\eps = 6\sum_{\sigma}\sigma$ (or equivalently $\eps = 6\Nm(k/\Q)$), then $\eps$ is of \textbf{Type 2}.
    \item If there exists an index 2 subgroup $H \subseteq G := \Gal(K/\Q)$ with $\Gal(K/k) \subseteq H$, $L:=K^H$ is imaginary quadratic and either 
    \begin{align*}
    \eps &= 12\sum_{\sigma \in \Sigma_L}\sigma \text{\quad or}\\
    \eps &= 12\sum_{\sigma \in \Sigma \setminus \Sigma_L}\sigma,
    \end{align*} then $\eps$ is of \textbf{Type 3 with field $L$}.
    \item If $\eps$ is not of Type 1, 2 or 3 then $\eps$ is \textbf{generic}.
\end{itemize}
\end{definition}

An equivalent definition of a Type 3 isogeny signature is that $k$ contains an imaginary quadratic field $L$ such that $\eps = 12\Nm(K/L)$ or $\eps = \rho\circ  12\Nm(K/L)$ where $\rho: L \to L$ is complex conjugation. Note that when given an explicit $\eps$ then the group $H$ can be retrieved as $H=\set{g \in \Gal(K/\Q) \mid g \circ \eps = \eps}$, the stabilizer of $\eps$. In particular, if $\eps$ is also Type 3 with field $L'$, then $L = L'$.

At this point it is instructive to recap what we are aiming to do in this Section, and how far towards this goal we have come. This will clarify what will be happening in the rest of the Section.

Our goal in this Section (stated at the outset) is to generalise and make Momose's isogeny classification theorem (Theorem A/Theorem 1 in \cite{momose1995isogenies}) algorithmic, as well as address the gaps and mistakes in his original proof. By generalise, we mean to remove the assumptions of $k$ being Galois over $\Q$, and also address primes $p$ that ramify in $k$.

Momose's isogeny classification theorem says that, for a given number field $k$, there is a constant $C_k$ such that, if $p$ is an isogeny prime for $k$ that is larger than $C_k$, then the associated isogeny character must be one of three Types, that in the Introduction we called Momose Types 1, 2 and 3. 

In this subsection, we have identified (in \Cref{def:sig_types}) three Types for the signature for which certain integers (the $A$, $B$ and $C$ integers from \Cref{def:ABC}) could be zero. Outside of these special Types of signature - that is, if the signature is generic - we know that these integers are nonzero; thus, by defining the integer
\[ ABC(\eps, \fq) := \lcm\left(\Nm(\fq), A(\eps, \fq), B(\eps, \fq), C(\eps, \fq)\right),\]
we may deal with the generic isogeny primes as a direct consequence of \Cref{cor:ABC_divs} and \Cref{prop:momose_lemma_2_non_galois}:

\begin{corollary}\label{cor:generic_summary}
Let $p \geq 17$ be an isogeny prime whose associated isogeny signature $\eps$ is generic. Then for all completely split prime ideals $\fq$, $p$ divides the nonzero integer $ABC(\eps, \fq)$.
\end{corollary}

Therefore, taking a completely split prime ideal $\fq$ in $k$, and taking the $\lcm$ of $ABC(\eps, \fq)$ across all of the generic signatures $\eps$, one would obtain a nonzero integer $C'$ such that, if $p$ is an isogeny prime for $k$ that does not divide $C'$, then the associated isogeny signature would be of Type 1, 2 or 3.

If it were true that the signature Types 1, 2 and 3 identified in \Cref{def:sig_types} coincided with Momose's Types 1, 2 and 3, then we would be done with our goal. This is however not the case; while an isogeny character of Momose Type $n$ (to be defined in the next subsection) implies that the signature is of Type $n$, the converse is not necessarily true (although it is true for $n = 1$).

What remains to be done, therefore, is to bound the prime degrees of isogenies whose signatures are of Type $n$ but that are not themselves of Momose Type $n$. If we can do this for $n = 2$ and $3$ (again this is not required for $n = 1$) then we have succeeded in making Momose's isogeny classification theorem algorithmic. We call this strategy \emph{going from signature Type $n$ to Momose Type $n$}. Specifically, this means that we will prove the following result.

\begin{theorem} Let $k$ be a number field. Then for $n = 2$ or $3$, there exist explicitly computable integers $B_n(K)$ such that if $p \nmid B_n(K)$ one has that if $\lambda$ is a $p$-isogeny character of signature Type $n$, then the isogeny character $\lambda$ is of Momose Type $n$. 
\end{theorem}

We refer to \Cref{cor:type_2_not_momose} and \Cref{prop:rocket_step_3} that respectively establish this result for $n = 2$ and $3$.

In the next subsection we will carry out this strategy, by studying the Type 1, 2 and 3 signatures more closely.

\subsection{From signature types to Momose Types}\label{ssec:from_sig_types_to_momose}
We emphasise that the Types identified in \Cref{def:sig_types} are properties of the \emph{signature} of an isogeny, while the Momose Types to be defined in the sequel are properties of the \emph{isogeny character} $\lambda$. The three special signature types identified in the previous section do not in all cases correspond exactly to Momose's three types identified for the isogeny character $\lambda$. They do, however, for Type 1 signatures, which \Cref{tab:david_types} shows.

\begin{definition}
We say that an isogeny character $\lambda$ is of \textbf{Momose Type $1$} if either $\lambda^{12}$ or $(\lambda/\chi_p)^{12}$ is everywhere unramified.
\end{definition}

Momose Type 1 primes will be handled in \Cref{sec:momose_type_1}.

\subsubsection{Signature Type 2}
We now consider isogenies of signature Type 2; that is, $\eps = 6\Nm(k/\Q)$. Observe from \Cref{tab:david_types} that this corresponds to $\mu$ and $\chi_p^6$ being the same up to an everywhere unramified character. The notion of Momose Type 2 goes further to say what this everywhere unramified character should be.

\begin{definition}
We say that an isogeny character $\lambda$ is of \textbf{Momose Type $2$} if $\lambda^{12} = \chi_p^6$.
\end{definition}

That is, the everywhere unramified character should be trivial. Note in this case that necessarily $p \equiv 3 \Mod{4}$, by the discussion appearing immediately after \Cref{cor:admissible_signatures}.

The following key result relates the notions of signature Type 2 and Momose Type 2.

\begin{proposition}\label{prop:type_2_to_momose_type_2}
Let $E/k$ be an elliptic curve admitting a $k$-rational $p$-isogeny of isogeny character $\lambda$ of signature $\eps$. If the following two conditions hold:
\begin{enumerate}
    \item $\eps$ is of Type 2;
    \item There exists a set of primes $\Gen$ generating $\Cl_k$, such that for all $\fq \in \Gen$ \begin{enumerate}
        \item $\fq$ is coprime to $p$;
        \item $\fq$ does not lie over the rational prime $2$;
        \item $E$ has potentially good supersingular reduction at $\fq$,
    \end{enumerate}
\end{enumerate}
then $\lambda$ is of Momose Type 2.
\end{proposition}

\begin{proof}
From (1) we have that $\mu\chi_p^{-6}$ is an everywhere unramified character, and hence defines an abelian extension of $k$ contained in the Hilbert class field of $k$, and thus is determined by its values at Frobenius automorphisms $\Frob_\fq$ for $\fq$ running through a set of generators of $\Cl_k$. Let $\fq \in \Gen$. From (2c) and from the discussion at the beginning of \Cref{subsec:divisibility}, we obtain $$\lambda(\Frob_\fq) \equiv \beta_\fq \Mod{\mathcal{P}^\fq}$$ for $\beta_\fq$ a supersingular Frobenius root over $\fq$ and for $\mathcal{P}^\fq$ a prime ideal above $p$ inside a field that is either $\Q$ or an imaginary quadratic field (this is the field denoted as $L^\fq$ in \Cref{subsec:divisibility}). By the same reasoning as in \Cref{rem:ss_frob_roots_waterhouse} we have that, for such $\beta_\fq$, that $\beta_\fq^{12} = \Nm(\fq)$ (this step requires $\fq$ to have odd characteristic), and hence
\begin{align*}
	\mu(\Frob_\fq) &= \Nm(\fq)^6 \Mod{p}\\
	 &= \chi_p(\Frob_\fq)^6.
	\end{align*}
Since this is true for all $\fq \in \Gen$, we obtain $\mu = \chi_p^6$ (i.e., not just equal up to an everywhere unramified character).
\end{proof}

\begin{corollary}\label{cor:type_2_not_momose}
Let $E/k$ be an elliptic curve admitting a $k$-rational $p$-isogeny of isogeny character $\lambda$ of signature $\eps$. Suppose that $\eps$ is of Type 2, but $\lambda$ is not of Momose Type 2. Then, for every set of prime ideal generators $\Gen$ of $\Cl_k$ of odd characteristic, $p$ divides the nonzero integer
\[ ABC_o(\Gen) := \underset{\fq \in \Gen}\lcm\left(A(\eps, \fq), B(\eps, \fq), C_o(\eps, \fq), \Nm(\fq)\right).\]
\end{corollary}

\begin{proof}
Since $\lambda$ is assumed not of Momose Type 2, then by \Cref{prop:type_2_to_momose_type_2}, there must exist a prime $\fq$ in $\Gen$ such that $E$ does \emph{not} have potentially good supersingular reduction at $\fq$; therefore $p$ must divide the integer $ABC_o(\Gen)$ by \Cref{cor:ABC_divs}. The proof therefore consists in showing that this integer is nonzero.

Let $\fq \in \Gen$. Since $\eps$ is of Type 2, we obtain that
\[ \gamma_\fq^\eps = \Nm(\fq)^{6h_\fq}. \]
This is clearly not equal to $1$, hence $A(\eps, \fq) \neq 0$. It is also clearly not equal to $\Nm(\fq)^{12h_\fq}$, whence $B(\eps, \fq) \neq 0$. Finally suppose that $C_o(\eps, \fq) = 0$. We would then obtain 
\[ \Nm(\fq)^{6h_\fq} = \beta_\fq^{12h_\fq}\]
for $\beta_\fq$ an ordinary Frobenius root over $\F_\fq$, which contradicts \Cref{cor:ord_frob_pow_not_rat}.
\end{proof}

One then deals separately with isogenies of Momose Type 2; this will be done in \Cref{sec:momose_type_2}. Note that when $h_k = 1$ then we can take the empty set of generators in \Cref{cor:type_2_not_momose}, and hence we obtain that an isogeny of signature Type 2 is automatically of Momose Type 2. The above \Cref{cor:type_2_not_momose} is the culmination of the strategy of going from signature Type 2 to Momose Type 2.

\subsubsection{Signature Type 3}
Finally we consider isogenies of signature Type 3 with field $L$. Note that \Cref{prop:momose_1_non_galois_md} only implies that the signature corresponding to an isogeny is unique for unramified primes greater than $13$. Thus, for ramified primes we have some freedom in the choice of signature that we can associate to an isogeny character in order to study it more precisely.

\begin{proposition}
Let $L\subseteq k$ be an imaginary quadratic field in which $p$ ramifies and let $\fp$ be the unique prime of $L$ above $p$. Then a $p$-isogeny character $\lambda$ defined over $k$ has a signature of Type 2 if and only if it has a signature of Type 3.
\end{proposition}
\begin{proof}
It suffices to prove that for all $\alpha \in k^\times$ coprime to $p$ we have
$$\Nm_{k/L}(\alpha)^{12} \Mod{\fp} \equiv  \Nm_{k/\Q}(\alpha)^6 \Mod{p} \equiv \overline{\Nm_{k/L}(\alpha)^{12}} \Mod{\fp}.$$
Since $p$ ramifies in the quadratic field $L$ we have $x \equiv \overline{x} \Mod{\fp}$ for all $x \in L$, and the proposition follows from the equality $\Nm_{k/\Q} = \Nm_{k/L} \cdot \overline{\Nm_{k/L}}$.
\end{proof}

The above proposition shows that when $p$ ramifies in $L$, there is no difference between a character having a signature of Type 2 and a signature of Type 3. Since we already have a strategy for studying isogeny characters with a Type 2 signature, we are reduced to studying isogenies with a Type 3 signature where $p$ is unramified in $L$, which we assume for the rest of this subsection.

We now define the notion of Momose Type 3.

\begin{definition}\label{def:momose_type_3}
We say that a $k$-rational $p$-isogeny character $\lambda$ is of \textbf{Momose Type $3$} if $k$ contains an imaginary quadratic field $L$ as well as its Hilbert class field, $p$ splits in $L$ as $\fp\cdot\overline{\fp}$, and for any prime $\fq$ of $k$ coprime to $\fp$,
\begin{align} \lambda^{12}(\Frob_\fq) = \alpha^{12} \Mod{\fp} \label{eqn:type3} \end{align}
for any $\alpha \in L^\times$ a generator of $\Nm_{k/L}(\fq)$.
\end{definition}

Since this definition is rather more involved than Momose Type 2, we describe how one may go from signature Type 3 to Momose Type 3 in three steps; these will be \Cref{cor:rocket_step_1}, \Cref{prop:rocket_step_2}, and \Cref{prop:rocket_step_3} below. We remark here that $p$ splitting in $L$ is automatic from $\eps$ being of signature Type 3, since the proof given at the end of the proof of \Cref{prop:momose_lemma_2_non_galois} is independent of $\fq$.

We take $\fp$ to be the prime of $L$ lying below our choice of $\fp_0$ in $K$. Notice that by changing the prime $\fp_0$ if necessary we may assume $\eps = 12\sum_{\sigma \in \Sigma_L}\sigma$.

\begin{lemma}\label{lem:type_3_nonzero}
Let $\eps$ be an isogeny signature of Type 3 with field $L \subseteq k$, and let $\fq$ be a prime ideal of $k$. Then the integers $A(\eps, \fq)$, $B(\eps, \fq)$ are nonzero. If in addition the ideal $\Nm_{k/L}(\fq)$ is not principal, then $C(\eps, \fq)$ is nonzero.
\end{lemma}

\begin{proof}
Write $\gamma_\fq$ for the generator of $\fq^{h_\fq}$. By definition of $\eps$, we obtain that, up to complex conjugation in $L$,
\[ \gamma_\fq^\eps = \Nm_{k/L}(\gamma_\fq)^{12} \]
and hence, by considering the ideals generated in $L$,
\[ \gamma_\fq^\eps\O_L = \Nm_{k/L}(\fq)^{12h_\fq}. \]
If $A(\eps, \fq) = 0$, then $\gamma_\fq^\eps = 1$ and we would obtain $\Nm_{k/L}(\fq)^{12h_\fq} = \O_L$ which is clearly not the case. If $B(\eps, \fq) = 0$, then $\gamma_\fq^\eps = \Nm(\fq)^{12h_\fq}$ and we would obtain $\Nm_{k/L}(\fq)^{12h_\fq} = (\Nm(\fq)\O_L)^{12h_\fq}$, taking the norm form $L$ to $\Q$ of this equation gives $\Nm(\fq)^{12h_\fq} = (\Nm(\fq)^2)^{12h_\fq}$ which cannot happen either. Finally, if $C(\eps, \fq) = 0$, then $\gamma_\fq^\eps = \beta_\fq^{12h_\fq}$ for a Frobenius root $\beta$ over $\F_\fq$, which implies that $\Nm_{k/L}(\fq) = \beta_\fq\O_L$ is principal.
\end{proof}

We then obtain Step 1 of the three step process from signature Type 3 to Momose Type 3.

\begin{corollary}[Step 1 of 3]\label{cor:rocket_step_1}
Let $\eps$ be an isogeny signature of Type $3$ with field $L \subseteq k$. Then either $k$ contains the Hilbert class field of $L$, or for any prime ideal $\fq$ of $k$ whose norm to $L$ is nonprincipal, we have that the integer $ABC(\eps, \fq)$ is nonzero; in particular, in this latter case, there are infinitely many such prime ideals $\fq$.
\end{corollary}

\begin{proof}
The norm on ideals induces a norm map $\Nm: \Cl_k \to \Cl_L$, let $N := \ker \Nm$ denote its kernel. If $N = \Cl_k$ then by class field theory $k$ contains the Hilbert class field of $L$ and we are done. So from now on assume  $N \neq \Cl_k$. Then by \Cref{lem:type_3_nonzero} $ABC(\fq, \eps)$ will be nonzero for all primes $\fq$ in $k$ that are in $\Cl_k \setminus N$, and since $\Cl_k \setminus N \neq \emptyset$ this set of primes is infinite.
\end{proof}


We are now reduced to the case of $k$ containing the Hilbert class field of $L$, and would like to show that \Cref{eqn:type3} is satisfied for all primes $\fq$ of $k$ coprime to $\fp$. Step 2 reduces this task to showing \Cref{eqn:type3} is satisfied for all prime ideals in a set of ideals generating $\Cl_k$. 

\begin{proposition}[Step 2 of 3]\label{prop:rocket_step_2}
Let $\lambda$ be a $k$-rational $p$-isogeny character of signature $\eps$. If the following conditions hold:
\begin{enumerate}
    \item $\eps=12\Nm_{k/L}$ is of Type 3 with field $L$;
    \item $k$ contains the Hilbert class field of $L$; \label{item:hilbert}
    \item \label{item:genset} There is a generating set $\Gen$ of $\Cl_k$ such that \cref{eqn:type3} is satisfied for all primes $\fq \in \Gen$;
\end{enumerate}
then $\lambda$ is of Momose Type 3.
\end{proposition}
\begin{proof}
Let $\fq$ be a prime of $k$. Then by (\ref{item:hilbert}) the ideal $\Nm_{k/L}(\fq)$ is principal. Let $\alpha \in L$ be a generator of $\Nm_{k/L}(\fq)$; then we need to show that $\lambda^{12}(\Frob_\fq) = \alpha^{12} \Mod{\fp}$.

Write $\fq_1,\ldots, \fq_n$ for the elements of $\Gen$ with $n = |\Gen|$. Since $\Gen$ generates the class group we can write \begin{align}\fq = \xi \prod_{i=1}^n \fq_i^{e_i}\label{eq:ideal_rel}\end{align} with $\xi \in k^\times$ and $e_1,\ldots, e_n$ positive integers. Again by (\ref{item:hilbert}) all the ideals $\Nm_{k/L}(\fq_i)$ are principal. Let $\alpha_i \in L^\times$ be a generator of $\Nm_{k/L}(\fq_i)$, then $\lambda^{12}(\Frob_{\fq_i}) = \alpha_i^{12} \Mod{\fp}$. This means:
\begin{align*}
    \lambda^{12}(\Frob_\fq) & \equiv \Nm_{k/L}(\xi)^{12} \prod \lambda^{12}(\Frob_{\fq_i}) \Mod{\fp_0} \\
    &\equiv \Nm_{k/L}(\xi)^{12}  \prod \alpha_i^{12} \Mod{\fp}
\end{align*}
where the first congruence follows from \Cref{prop:momose_1_non_galois_md}.
Now applying $\Nm_{k/L}$  to \Cref{eq:ideal_rel} gives that $\alpha' := \Nm_{k/L}(\xi)^{12}  \prod \alpha_i^{12}$ is a generator of $\Nm_{k/L}(\fq)$ so in particular \Cref{eqn:type3} holds for the generator $\alpha'$ of $\Nm_{k/L}(\fq)$. Now since $\alpha$ and $\alpha'$ both generate $\Nm_{k/L}(\fq)$ we have $\alpha/\alpha' \in L^\times$ is a unit. Since $L$ is imaginary quadratic this unit has order dividing 12 and hence $(\alpha/\alpha')^{12}=1$. In particular  $\lambda^{12}(\Frob_\fq) \equiv \alpha'^{12} \equiv \alpha^{12}  \Mod{\fp}$. 
\end{proof}

The final step is then to show that we can arrange for Condition (\ref{item:genset}) above to be satisfied. We first define
\begin{equation}
    \label{eqn:C_int_non_zero}
    C^\ast(\eps, \fq) = \lcm(\left\{ \Nm_{K(\beta)/\Q}(\gamma_\fq^\varepsilon - \beta^{12h_\fq}) \ | \ \beta \in \overline{k} \mbox{ is a Frobenius root over }\F_\fq\right\}),
\end{equation}
where in the $\lcm$ we only take nonzero terms. Note that this may be considered the necessarily-non-zero part of the integer $C(\eps, \fq)$, and the notation has been chosen to reflect this.

\begin{proposition}[Step 3 of 3]\label{prop:rocket_step_3}
Let $\lambda$ be a $k$-rational $p$-isogeny character of signature $\eps$ of Type 3 with field $L$. Assume moreover that $p \geq 17$ and is unramified in $L$, and that $k$ contains the Hilbert class field of $L$. Let $\Gen$ be a set of prime ideal generators of $\Cl_k$ of odd residue characteristic. Then, either $p$ divides the nonzero integer
\[ ABC^\ast(\eps, \Gen) := \underset{\fq \in \Gen}\lcm\left(A(\eps, \fq), B(\eps, \fq), C^\ast(\eps, \fq), \Nm(\fq)\right);\]
or, for all $\fq \in \Gen$, writing $\alpha_\fq$ for a generator of $\Nm_{k/L}(\fq)$, we have that
\[ \lambda^{12}(\Frob_\fq) \equiv \alpha_\fq^{12} \Mod{\fp}\]
and thus $\lambda$ is of Momose Type 3.
\end{proposition}

\begin{proof}
Let $\fq \in \Gen$, and write $q$ for the rational prime under $\fq$. Since $k$ contains the Hilbert class field of $L$ one has that $\Nm_{k/L}(\fq)$ is principal, so write $\alpha_{\fq}\O_L = \Nm_{k/L}(\fq)$. Let $h = h_{\fq}$ be the order of $\fq$ in $\Cl_k$ and $(\gamma) = \fq^{h}$ be a principal generator.

Since $p \geq 17$ and is assumed unramified in $L$, one shows in exactly the same way as in the final assertion of the proof of \Cref{prop:momose_lemma_2_non_galois} that $p$ splits in $L$. By applying Proposition 4.9 of \cite{david2012caractere} (which is just using her Proposition 2.4 together with the fact that $p$ splits in $L$) to the principal ideal $(\gamma)$ we obtain
\[ \mu(\Frob_{\fq})^{h} \equiv \Nm_{k/L}(\gamma)^{12} \Mod{\fp}. \]
Since $\Nm_{k/L}(\gamma)/\alpha_\fq^h$ is a unit in $\O_L$ we get $(\Nm_{k/L}(\gamma)/\alpha_\fq^h)^{12}=1$ and we may rewrite the above as
\[ \mu(\Frob_{\fq})^{h} \equiv \alpha_{\fq}^{12h} \Mod{\fp}; \]
but for the $h$\textsuperscript{th} power, this is what we want to obtain, and the rest of the proof is about how to gracefully take the $h$\textsuperscript{th} root.

If $E$ has potentially multiplicative reduction at $\fq$, then $p$ divides either $A(\eps, \fq)$ or $B(\eps, \fq)$, both of which are nonzero by \Cref{lem:type_3_nonzero}. Thus, for $p$ not dividing $\lcm(A(\eps, \fq), B(\eps, \fq)) $, we have that $E$ has potentially good reduction at $\fq$, and so there exists a Frobenius root $\beta_\fq \in L^\fq$ of an elliptic curve over $\F_\fq$ and a prime ideal $\fp'$ of $L^\fq$ lying over $p$ such that
\begin{align*}
\lambda(\Frob_\fq) &\equiv \beta_\fq \Mod{\fp'} \\
\Rightarrow \mu(\Frob_\fq) &\equiv \beta_\fq^{12} \Mod{\fp'}.
\end{align*}

We therefore obtain an equality in $\F_p^\times$:
\[ \beta_\fq^{12h} \Mod{\fp'} = \alpha_\fq^{12h} \Mod{\fp}.\]
Suppose we had actual equality in characteristic $0$:
\[ \beta_\fq^{12h} = \alpha_\fq^{12h}. \]
There are two ways this equation could be satisfied: either $L = L^\fq$, or $L \neq L^\fq$ and this equality is between rational numbers. By \Cref{cor:ord_frob_pow_not_rat} and \Cref{cor:ss_frob_pow_rat}, the second case can only occur for supersingular values of $\beta_\fq$.

In the case $L = L^\fq$ we take the $12h$\textsuperscript{th}-root to obtain $\beta_\fq = \zeta\cdot \alpha_\fq$ for $\zeta \in L$ a $12h$\textsuperscript{th}-root of unity. However, since $L$ is imaginary quadratic, all roots of unity are either second, fourth, or sixth roots of unity, whence we obtain $\beta_\fq^{12} = \alpha_\fq^{12}$ and conclude 
\[ \mu(\Frob_\fq) \equiv \alpha_\fq^{12} \Mod{\fp} \]
as required.

In the case $L \neq L^\fq$ we have that $\beta_\fq$ is supersingular and we argue as follows. The definition of $\alpha_\fq$ yields the equality of $L$-ideals
\[ \Nm_{k/L}(\fq)^{12h} = \alpha_\fq^{12h}\O_L.\]
Since $\alpha_\fq^{12h} \in \Q$, the $\Gal(L/\Q)$-action acts trivially on it and hence also on $\Nm_{k/L}(\fq)^{12h}$ and therefore also on $\Nm_{k/L}(\fq)$. Let $\fq'$ be the prime of $L$ lying below $\fq$ and $f$ the integer such that $\Nm_{k/L}(\fq) = (\fq')^f$, then applying \Cref{prop:ideal_generators} to $(\fq')^{f}$ gives $\alpha_\fq^{12}=\Nm_{L/\Q}(\fq')^{6f}=\Nm_{L/\Q}(\Nm_{k/L}(\fq)^6)=\Nm_{k/\Q}(\fq)^6$.

On the other hand, since $\beta_\fq$ is a supersingular Frobenius root, we have from \Cref{cor:ss_frob_pow_norm} that $\beta_\fq^{12} = \Nm(\fq)^6$ (using that $\fq$ has odd residue characteristic). Combining these we get $\alpha_\fq^{12} = \beta_\fq^{12} $ and again conclude 
\[ \mu(\Frob_\fq) \equiv \alpha_\fq^{12} \Mod{\fp} \]
as required. The final claim that $\lambda$ is of Momose Type 3 now follows from \Cref{prop:rocket_step_2}.

To summarise, if we had actual equality in characteristic zero between $\alpha_\fq^{12h}$ and $\beta_\fq^{12h}$, then we are finished. Of course, we cannot in general jump from an equality in characteristic $p$ to one in characteristic $0$; but we may do so for $p$ outside the support of an explicitly computable integer. Namely, if $\beta_\fq^{12h} \neq \alpha_\fq^{12h}$, then $p$ would divide the nonzero integer $\Nm_{L(\beta)/\Q}\left(\beta_\fq^{12h} - \alpha_\fq^{12h}\right)$. Since we need to cover all possible roots $\beta_\fq$, and observing that we may take $\alpha_\fq^{12h}$ to be $\gamma_\fq^\eps$, we obtain equality in characteristic zero outside of the integer $C^\ast(\eps, \fq)$ defined in \Cref{eqn:C_int_non_zero}. Doing this for all $\fq \in \Gen$ yields the integer defined in the statement of the Proposition.
\end{proof}

In conclusion, we recap this treatment of going from signature Type 3 to Momose Type 3. Signature Type 3 requires only that $k$ contain an imaginary quadratic field, whereas Momose Type 3 requires further that $k$ contain the Hilbert class field of $L$. Step 1 of the above process (\Cref{cor:rocket_step_1}) identifies a nonzero integer, outside of which an isogeny character of signature Type 3 must have $k$ containing the Hilbert Class field of $L$. However, a Momose Type 3 isogeny character requires even more than just $k$ containing the Hilbert Class field of $L$, so Step 3 (\Cref{prop:rocket_step_3}) identifies a nonzero integer outside of which the isogeny character is indeed of Momose Type 3. Taking both of these nonzero integers gives us a multiplicative bound, outside of which an isogeny character of signature Type 3 must be of Momose Type 3. (This is the integer $D(\eps_3)$ in \Cref{alg:main} in the next subsection.)

\subsection{The algorithmic version of Momose's isogeny theorem}

We summarise the previous subsections and present the algorithmic version of Momose's isogeny classification theorem. \Cref{alg:main} below takes as input a number field $k$, and outputs an integer $\MMIB(k)$, referred to as the \textbf{Momose Multiplicative Isogeny Bound} of $k$. Before showing the details of the algorithm, we present the main result concerning it, which is the main result of the paper.

\begin{theorem}\label{thm:main_body}
Let $k$ be a number field. Then the integer $\MMIB(k)$ output by \Cref{alg:main} is nonzero. Furthermore, if there exists an elliptic curve over $k$ admitting a $k$-rational $p$-isogeny of isogeny character $\lambda$, with $p \nmid \MMIB(k)$, then $\lambda$ is of Momose Type 1, 2 or 3.
\end{theorem}

For the reader's convenience we will collect the relevant results from this Section into a unified proof of the above theorem. Before we can do that, however, we need to present the algorithm that computes $\MMIB(k)$. This depends on the choice of two ``auxiliary sets'' which are declared at the outset of the following.

\begin{algorithm}\label{alg:main}
Given a number field $k$, compute an integer $\MMIB(k)$ as follows.

\begin{enumerate}
    \item Choose a finite set $\Aux$ of prime ideals of $k$ which contains at least one totally split prime; if the class number $h_k$ of $k$ is greater than one, and $k$ contains an imaginary quadratic field $L$ but not its Hilbert class field, then $\Aux$ is also required to contain at least one prime ideal of $k$ whose norm to $L$ is nonprincipal.
    \item Choose a finite set $\AuxGen$ of sets $\Gen$ of prime ideal generators for the class group $\Cl_k$ of $k$ of odd characteristic.
    \item Make the following definitions.
    \begin{align*}
K &= \mbox{Galois closure of } k \mbox{ over } \Q \\
\fq &= \mbox{a prime ideal of $k$}\\
\F_\fq &= \mbox{residue field of $\fq$, of degree $f$}\\
S &= \mbox{the generic isogeny signatures for $k$ (see \Cref{def:sig_types})}\\
\eps_6 &= \mbox{the Type 2 signature $(6,\ldots, 6)$}\\
A(\varepsilon, \fq) &= \Nm_{K/\Q}(\gamma_\fq^\varepsilon - 1)\\
B(\varepsilon, \fq) &= \Nm_{K/\Q}(\gamma_\fq^\varepsilon - \Nm(\fq)^{12h_\fq})\\
C_s(\varepsilon, \fq) &= \begin{cases*} 
B(2\varepsilon, \fq) & \mbox{if $|\F_\fq| = 2^f$ with $f$ odd}\\
\Nm_{K/\Q}(\gamma_\fq^\varepsilon - \Nm(\fq)^{6h_\fq}) & \mbox{otherwise}
\end{cases*}\\
C_o(\varepsilon, \fq) &= \lcm\left(\left\{ \Nm_{K(\beta)/\Q}(\gamma_\fq^\varepsilon - \beta^{12h_\fq}) \ | \ \mbox{\parbox{4cm}{$\beta$ is an ordinary \\ Frobenius root over $\F_\fq$} }\right\}\right)\\
C(\varepsilon, \fq) &= \lcm\left(C_o(\varepsilon, \fq), C_s(\varepsilon, \fq)\right)\\
ABC(\varepsilon, \fq) &= \mbox{$\lcm(A(\varepsilon, \fq), B(\varepsilon, \fq), C(\varepsilon, \fq), \Nm(\fq))$}\\
ABC_o(\Gen) &= \underset{\fq \in \Gen}\lcm\left(A(\eps_6, \fq), B(\eps_6, \fq), C_o(\eps_6, \fq), \Nm(\fq)\right)\\
\lcm^\ast &= \mbox{an $\lcm$ over only nonzero terms in a given set}\\
C^\ast(\eps, \fq) &= \lcm^\ast\left(\left\{ \Nm_{K(\beta)/\Q}(\gamma_\fq^\varepsilon - \beta^{12h_\fq}) \ | \ \mbox{\parbox{2.9cm}{$\beta$ is a Frobenius \\ root over $\F_\fq$} }\right\}\right)\\
ABC^\ast(\eps, \Gen) &= \underset{\fq \in \Gen}\lcm\left(A(\eps, \fq), B(\eps, \fq), C^\ast(\eps, \fq), \Nm(\fq)\right)\\
\GenericBound(k) &= \underset{\eps \in S}\lcm \left( \gcd_{\fq \in \Aux}(ABC(\varepsilon, \fq))\right)\\
\TypeTwoNotMomoseBound(k) &= \gcd_{\Gen \in \AuxGen}\left(ABC_o(\Gen)\right)\\
S_3 &= \mbox{the Type 3 signatures for $k$}\\
L(\eps_3) &= \mbox{the imaginary quadratic field corresponding to }\eps_3 \in S_3\\
\Delta_{L(\eps_3)} &= \mbox{the discriminant of }L(\eps_3)\\
\HCF(L(\eps_3)) &= \mbox{the Hilbert class field of }L(\eps_3)\\
D(\eps_3) &= \begin{cases*} 
\underset{\fq \in \Aux}{\gcd}\left(ABC(\varepsilon_3, \fq)\right) & \mbox{if $\HCF(L(\eps_3)) \not\subseteq k$}\\
\underset{\Gen \in \AuxGen}{\gcd}\left(ABC^\ast(\eps_3, \Gen)\right) & \mbox{otherwise}
\end{cases*}\\
\TypeThreeNotMomoseBound(k) &= \begin{cases*}  
\underset{\eps \in S_3}\lcm\left(D(\eps)\right) & \mbox{if $\HCF(L(\eps_3)) \not\subseteq k$}\\
\underset{\eps \in S_3}\lcm\left(D(\eps),\Delta_{L(\eps)}\right) & \mbox{otherwise}
\end{cases*}\\
\MMIB(k) &= \lcm\left(\begin{array}{c}
\GenericBound(k),\\ \TypeTwoNotMomoseBound(k),\\
\TypeThreeNotMomoseBound(k)\end{array}\right).
\end{align*}
\item Return  $\MMIB(k)$.
\end{enumerate}
\end{algorithm}

Note that in the above, contrary to the previous sections, we are taking several prime ideals $\fq$ and several generating sets $\Gen$ in items (1) and (2); this allows us to take multiplicative sieves and thereby get a smaller $\MMIB(k)$; this explains the $\gcd$s appearing in item (3) of the above algorithm.

We now prove \Cref{thm:main_body}.

\begin{proof}[Proof of \Cref{thm:main_body}]
Let $E/k$ be an elliptic curve admitting a $k$-rational $p$-isogeny of isogeny character $\lambda$. We suppose that $\lambda$ is not of Momose Type 1, 2 or 3. We wish to show that $p \mid \MMIB(k)$.

Write $\eps$ for the signature of $\lambda$. If $\eps$ is generic, then by \Cref{cor:generic_summary} we have that $p$ divides the nonzero integer $ABC(\eps, \fq)$ for all completely split prime ideals $\fq$; and hence, by definition of $\Aux$, $p$ divides the nonzero integer $\GenericBound(k)$.

If $\eps$ is of Type 1, then $\lambda$ is of Momose Type 1, which we have precluded. If $\eps$ is of Type 2, then we are in the setup of \Cref{cor:type_2_not_momose}, and we obtain that $p$ must divide the nonzero integer $\TypeTwoNotMomoseBound(k)$. 

Finally, if $\eps$ is of Type 3 with field $L$, then we condition on whether or not $\HCF(L) \subseteq k$. If this is not the case, then we argue as follows. From \Cref{cor:ABC_divs} we have that $p$ divides $\underset{\fq \in \Aux}{\gcd}\left(ABC(\varepsilon_3, \fq)\right)$; to show that this integer is nonzero we use \Cref{cor:rocket_step_1}, observing that this is precisely why we included in $\Aux$ at least one prime ideal of $k$ whose norm to $L$ is nonprincipal. On the other hand, if $\HCF(L) \subseteq k$, then from \Cref{prop:rocket_step_3} we have that $p$ divides $D(\eps)$ or $p$ divides $\Delta_{L(\eps)}$. In both cases we obtain that $p$ divides the nonzero integer $\TypeThreeNotMomoseBound(k)$.

Thus, in all cases, we obtain that $p$ divides the nonzero integer $\MMIB(k)$ as required.
\end{proof}

\section{\GenericBound}\label{sec:generic}

This section deals with some aspects of the implementation of $\GenericBound(k)$, particularly regarding how it may be optimised for speed and for obtaining smaller multiplicative bounds. Names of functions refer to functions defined in \path{sage_code/generic.py} unless otherwise specified.

\subsection{Unit precomputation}

We apply \Cref{prop:momose_1_non_galois_md} in the case where $\alpha$ is a unit in $k$, from which we get the following immediate corollary.

\begin{corollary}\label{cor:unit_bound}
Let $\lambda$ be a $p$-isogeny character over $k$ of signature $\eps$ and $\alpha \in \O_k^\times$ a unit in $k$. Then $p$ divides $\Nm_{K/\Q}(\alpha^\eps-1)$.
\end{corollary}

The following result shows that there are many situations in which $\Nm_{K/\Q}(\alpha^\eps-1)$ is guaranteed to be nonzero for some unit $\alpha$ and hence the divisibility bound on $p$ from the above corollary is nontrivial.

\begin{proposition}
Let $k$ be a totally real number field of degree $d$ and $\eps=\sum_{\sigma}a_\sigma\sigma$ a $k$-isogeny signature such that $\alpha^\eps=1$ for all $a \in \O_k^\times$. Then all the integers $a_\sigma$ are the same.
\end{proposition}
\begin{proof}
This is essentially the same proof as Lemma 3.3 of \cite{freitas2015criteria}. In this proof (and only in this proof) we override the $\lambda$ and $\mu$ notation from the rest of the paper to accord with the notation in \cite{neukirch}. Write $\sigma_1, \ldots, \sigma_d$ for the embeddings of $k$ in $K$, and write $a_i$ for $a_{\sigma_i}$. We consider the basic map from the setup of Dirichlet's Unit Theorem:
\begin{align*}
    \lambda : \O_k^\times &\rightarrow \R^{d}\\
    u &\mapsto \left(\log|\sigma_1(u)|, \ldots, \log|\sigma_{d}(u)| \right),
\end{align*}
and we have the following facts about $\lambda$ (see for example Chapter 1, Section 7 of \cite{neukirch}):
\begin{itemize}
    \item The image of $\lambda$ is contained in the \emph{trace-zero hypersurface}
    \[ H := \left\{x \in \R^d : \Tr(x) = 0\right\}, \]
    and moreover is a complete lattice inside $H$;
    \item The kernel of $\lambda$ is isomorphic to the roots of unity $\mu(k)$ inside $k$, which in our case of $k$ totally real is just $\left\{\pm1\right\}$.
\end{itemize}
We therefore obtain that $\O_k^\times/\left\{\pm1\right\}$ maps isomorphically to a complete lattice of dimension $d-1$ under $\lambda$.

By our assumption, we also have that $\O_k^\times$ is contained in the hypersurface $$a_1x_1 + \cdots + a_dx_d = 0.$$ If not all of the $a_i$ are the same, then this is a different hypersurface to $H$, and therefore $\O_k^\times$ would be contained in two distinct hypersurfaces; it would thus be contained in their intersection, which would be of dimension $d-2$. It is then clear that $\O_k^\times/\left\{\pm1\right\}$ could not be a complete lattice of dimension $d-1$, yielding the desired contradiction.
\end{proof}

This observation is useful in practice to quickly obtain a bound on isogeny primes of signature $\eps$, since it only requires the computation of a few norm values related to generators of the unit group $\O_k^\times$ which are typically very fast. As such, the resulting divisibilities are computed as an initial step in the main routine, in the function \path{get_U_integers}.

\subsection{Choosing auxiliary primes}
The set $\Aux$ of auxiliary primes is of central importance in the algorithm to compute $\GenericBound(k)$, and there are two strategies for constructing $\Aux$ implemented in \emph{Isogeny Primes}.

The first is to take all prime ideals up to a bound on the norm which the user may specify (default: $50$). The resulting divisibility conditions are not guaranteed to be nontrivial, as this set might not contain a completely split prime of $k$ (which ensures that the integer $ABC(\eps, \fq)$ is nonzero); if it does not contain a completely split prime, then we simply add one in. This possibly newly added prime ideal is referred to in the code as an \textbf{emergency auxiliary prime}.

The default strategy, however, is an \emph{auto-stop strategy} which successively takes only completely split prime ideals as auxiliary primes, computes the successive integers $\GenericBound(k)$, and terminates when a certain number of them (default: $4$) are the same. This is the default behaviour since it was observed during testing that this results in faster runtime whilst not compromising on the tightness of the multiplicative bound.

\subsection{Optimising signatures}

For a number field $k$ of degree $d$, the number of possible $k$-isogeny signatures $\eps$ to be considered (including the Type 1 and 2 signatures) is $5^d$, which becomes prohibitively large very quickly. \Cref{cor:optimum_eps} below allows one to reduce the number of signatures to be considered by a factor of $2d$ in the Galois case, and by a factor of $2$ otherwise. We first establish the following lemma.

\begin{lemma}\label{lem:gal_act_eps}
Let $E/k$ be an elliptic curve over a number field admitting a $k$-rational $p$-isogeny with isogeny character of signature $\eps$. Then, for $\sigma \in \Aut(k/\Q)$, the conjugate elliptic curve $\sigma(E)$ admits a $k$-rational $p$-isogeny of signature $\eps \circ \sigma$.
\end{lemma}

\begin{proof}
Write $H = \Aut(k/\Q)$ and $\eps = \sum_{\tau \in \Hom(k,K)}a_\tau\tau$, and let $\lambda$ be the isogeny character with signature $\eps$. Write as before $\mu = \lambda^{12}$, and identify (as in \Cref{prop:momose_1_non_galois_md}; see also \Cref{rem:mu_galois_and_ideal}) $\mu$ as an $\F_p^\times$-valued character on the group $I_k(p)$ of ideals of $k$ coprime to $p$. Fix a prime ideal $\fp$ of $K$ above $p$. From the natural $\Aut(k/\Q)$-action on $I_k(p)$ we obtain, for $\alpha \in k^\times$ coprime to $p$ that
\begin{align*}
    \mu^\sigma((\alpha)) &= \mu((\sigma(\alpha)))\\
    &\equiv (\sigma(\alpha))^\eps \Mod{\fp}\\
    &\equiv \prod_{\tau \in G}\tau(\sigma(\alpha))^{a_\tau} \Mod{\fp}\\
    &\equiv \prod_{\tau \in G}\tau(\alpha)^{a_{\tau\sigma^{-1}}} \Mod{\fp}\\
    &\equiv \alpha^{\eps \circ \sigma} \Mod{\fp}
\end{align*}
whence the result follows.
\end{proof}

For $\eps = \sum_{\sigma}a_\sigma \sigma$, we write $12 - \eps$ for the isogeny signature $\sum_{\sigma}(12 - a_\sigma) \sigma$.

\begin{corollary}\label{cor:optimum_eps}
Let $\lambda$ be a $p$-isogeny character over $k$ of signature $\eps$ and $\fq$ a prime of $k$ coprime to $p$.

\begin{enumerate}
    \item If $p$ divides $ABC(\eps, \fq)$, then $p$ also divides $ABC(12 - \eps, \fq)$.
    \item For $\sigma \in \Aut(k/\Q)$, if $p$ divides $ABC(\eps, \fq)$, then $p$ also divides $ABC(\eps \circ \sigma, \fq)$.
\end{enumerate}
\end{corollary}

\begin{proof}
For $\lambda$ a $p$-isogeny character of signature $\eps$, we have from \cite[Remark 2]{momose1995isogenies} that the dual isogeny has character $\chi_p\lambda^{-1}$ and is of signature $12 - \eps$; since the dual isogeny also has degree $p$, this proves (1), and (2) similarly follows from \Cref{lem:gal_act_eps}.
\end{proof}

The code which generates the minimum set of $\eps$ tuples to be considered (i.e. modding out by the $\eps \mapsto 12 - \eps$ and $\Aut(k/\Q)$ actions) is implemented in the function \path{generic_signatures}. This returns a dictionary with keys the minimum set of signatures, and values the \emph{Type} of that signature, as defined (for all signatures, not just the generic ones) in \Cref{tab:eps_types}.

\begin{table}[htp]
\begin{center}
\begin{tabular}{|c|c|c|}
\hline
Type & Description & Examples\\
\hline
Type 1 & All $a_i = 0$ or all $a_i = 12$ & $(0,0,0), (12,12,12)$\\
\hline
\multirow{2}{*}{Quadratic nonconstant} & All $a_i \in \left\{0,12\right\}$ & $(0,12,0)$\\
& Not all identical & $(12,12,0)$\\
\hline
Sextic constant & All $a_i = 4$ or all $a_i = 8$ & $(4,4,4), (8,8,8)$\\
\hline
\multirow{3}{*}{Sextic nonconstant} & All $a_i \in \left\{0,4,8,12\right\}$ & $(4,0,4)$\\
& Not all identical & $(4,8,8)$\\
& At least one $a_i$ is $4$ or $8$ & $(0,0,4)$\\
\hline
Type 2 & All $a_i = 6$ & $(6,6,6)$\\
\hline
\multirow{3}{*}{Quartic nonconstant} & All $a_i \in \left\{0,6,12\right\}$ & $(0,6,0)$\\
& Not all identical & $(12,6,0)$\\
& At least one $a_i$ is $6$ & $(6,6,0)$\\
\hline
Mixed & Tuple contains $6$ and either $4$ or $8$ & $(4,6,0), (8,0,6)$\\
\hline
\end{tabular}
\vspace{0.3cm}
\caption{\label{tab:eps_types}The seven signature types. $a_i$ refers to the integers in the signature. The examples show possible signatures for each Type arising from a non-Galois cubic number field.}
\end{center}
\end{table}

The reason for grouping all signatures into these particular groups is explained in \Cref{sec:filtering}.

\begin{remark}\label{rem:not_passing_to_gal_closure_is_faster}
Since the number of signatures to be considered is exponential in the degree of $k$, passing to the Galois closure of $k$ over $\Q$ - which was Momose's method of proof of his isogeny theorem - incurs a significant cost to the algorithm. This was one of our main motivations for seeking a method for bounding $k$-isogenies of prime degree which did not pass to the Galois closure, which led to removing the Galois assumptions in Momose's Lemmas 1 and 2 (which are \Cref{prop:momose_1_non_galois_md} and \Cref{prop:momose_lemma_2_non_galois} in \Cref{sec:strengthen_momose}).
\end{remark}

\section{$\TypeOneBound$}\label{sec:momose_type_1}

In this section we treat isogenies of signature Type $1$; that is, $\eps = (0,\ldots,0)$ or $(12,\ldots,12)$. As mentioned in the proof of \Cref{cor:optimum_eps}, if an elliptic curve admits an isogeny with character of signature $(12,\ldots,12)$, then the dual isogeny will be of signature $\eps = (0,\ldots,0)$, so for our purposes of bounding the possible isogeny primes of this signature, we may assume without loss of generality throughout this section that $\eps = \eps_0 = (0,\ldots,0)$. 

The main goal of this section is to explain how the following algorithm produces a multiplicative bound on the isogeny primes of Momose Type 1.

\begin{algorithm}\label{alg:type_1_primes}
Given a number field $k$ of degree $d$, compute an integer $\TypeOneBound(k)$ as follows.
\begin{enumerate}
    \item Choose a finite set $\Aux$ of odd rational primes $q$.
    \item Make the following definitions.
    \begin{align*}
    \BadFormalImmersion(d) &= \mbox{\parbox{10cm}{product of the bad formal immersion primes in degree $d$ (see \Cref{def:bad_formal_immersion_sets} and \Cref{thm:explicit_formal_immersion});}}\\[1ex]
    \AGFI_d(q) &= \mbox{\parbox{10cm}{product of the almost good formal immersion primes in degree $d$ which are bad formal immersion primes in characteristic $q$ (see \Cref{def:bad_formal_immersion_sets} and \Cref{thm:explicit_formal_immersion});}}\\[1ex]
\eps_0 &= \mbox{the Type 1 signature $(0,\ldots, 0)$;}\\
D(\Aux) &= \gcd_{\substack{\fq | q\\q \in \Aux}} \left( \lcm(B(\eps_0, \fq), C(\eps_0, \fq), \Nm(\fq), \AGFI_d(q))\right) \\[0.5ex]
&\ \mbox{(see \Cref{def:ABC} for $B$ and $C$);}\\
\TypeOneBound(k) &= \lcm\left(\BadFormalImmersion(d), D(\Aux)\right).
\end{align*}
\item Return  $\TypeOneBound(k)$.
\end{enumerate}
\end{algorithm}

Specifically, the main result of this section is \Cref{thm:type_1_brief} from the introduction.

\begin{theorem}\label{thm:type_1_main}
Let $k$ be a number field. Then $\TypeOneBound(k)$ is nonzero. Moreover, if $p$ is an isogeny prime for $k$ whose associated isogeny character is of Momose Type $1$, then $p$ divides $\TypeOneBound(k)$.
\end{theorem}

As one may see from the definitions made in \Cref{alg:type_1_primes}, \Cref{thm:type_1_main} is concerned with the notion of \textbf{good and bad formal immersion primes} of a given degree. We therefore explain these notions before commencing with the proof of \Cref{thm:type_1_main}.

For a scheme $X$, let $\O_X$ denote the structure sheaf, $\O_{X,x}$ the local ring at a point $x$ of $X$, and $\widehat{O}_{X,x}$ the completion of the local ring with respect to its maximal ideal $\mathfrak{m}_{X,x}$. Following \cite[\S3]{mazur1978rational}, we make the following definition.

\begin{definition}
If $f : X \to Y$ is a morphism of finite type between noetherian schemes, we say that $f$ is a \textbf{formal immersion} at a point $x$ if the induced map on the completion of local rings
\[ \widehat{O}_{Y,f(x)} \to \widehat{O}_{X,x}\]
is surjective.
\end{definition}

We will take $X$ to be the $d$\textsuperscript{th} symmetric power (for $d \geq 1$ an integer) of the modular curve $X_0(p)$ (for $p$ prime), which we write as $X_0(p)^{(d)}$, and which we regard as a smooth scheme over $S := \Spec(\Z[1/p])$. We then have a map
\begin{alignat*}{2}
    f^{(d)}_{p} : \eqmathbox{X_0(p)^{(d)}_{/S}} &\longrightarrow \eqmathbox{J_0(p)_{/S}} & &\longrightarrow \eqmathbox{J_{e,/S}}\\
    \eqmathbox{D} &\longmapsto \eqmathbox{[D - d(\infty)]} & &\longmapsto [D - d(\infty)] \Mod{\gamma_{\mathfrak{J}}J_0(p)};
\end{alignat*}
here $\mathfrak{J}$ is the winding ideal (see Chapter II, Section 18 of \cite{Mazur3}) and $J_e$ is the winding quotient of $J_0(p)$, the largest rank zero quotient of $J_0(p)$ assuming the Birch-Swinnerton-Dyer conjecture.

The point $x \in X$ at which we consider the formal immersion notion will be the $S$-section $\infty^d := (\infty, \ldots, \infty)$; however, we will only ever be concerned with a `local' notion of formal immersion; that is, we say that $f^{(d)}_{p}$ is a formal immersion at $\infty^d$ in characteristic $q$ (for $q \neq p$) if it is a formal immersion when the schemes are considered over the base $\Spec(\Z_{(q)})$.

Merel proved \cite[Proposition 3]{merel1996bornes}, for $p$ sufficiently large with respect to $d$, that the map $f^{(d)}_{p}$ \emph{is} in fact a formal immersion along $\infty^d$ in characteristic $3$, and even gave an explicit bound on $p$, which was subsequently improved by Parent\footnote{Pierre Parent was also a student of Bas Edixhoven.} \cite[Theorem 1.8 and Proposition 1.9]{parent1999bornes} and Oesterl\'{e}, the latter of whom never published his bound, but which may be found explained in Section 6 of \cite{derickx2019torsion}. 

We however take a more algorithmic and general approach which works for all odd primes $q$ (not just equal to $3$; we will in fact be able to gain some information from considering $q = 2$ as well; see \Cref{rem:formal_immersion_at_2}). To explain our strategy, we make the following definitions.

\begin{definition}\label{def:bad_formal_immersion_sets}
Let $d \geq 1$ be an integer.

\begin{enumerate}
    \item
    For $p$ and $q$ distinct primes, we say that $p$ is a \textbf{bad formal immersion prime in degree $d$ at characteristic $q$} if the map $f_p^{(d)}$ is not a formal immersion in characteristic $q$.    
    \item 
    We say that $p$ is a \textbf{bad formal immersion prime in degree $d$} if the map $f_p^{(d)}$ is not a formal immersion in characteristic $q$ for infinitely many primes $q$.
    \item
    We say that $p$ is an \textbf{almost good formal immersion prime in degree $d$} if the map $f_p^{(d)}$ is not a formal immersion in characteristic $q$ for at least one but at most finitely many primes $q\neq 2, p$.
    \item
     We say that $p$ is a \textbf{good formal immersion prime in degree $d$} if the map $f_p^{(d)}$ \emph{is} a formal immersion in characteristic $q$ for all primes $q \neq 2, p$.
\end{enumerate}
\end{definition}

The last definition here might be more accurately described as being \textbf{good outside $2$ and $p$}, but for simplicity we prefer to have this implicit. Also, since $d$ is usually fixed, we often drop the `in degree $d$' qualifier.


Our strategy is based on the following result.


\begin{theorem}\label{thm:explicit_formal_immersion}
Let $d \geq 1$ be an integer.
\begin{enumerate}
    \item The set of bad formal immersion primes in degree $d$ is finite and effectively computable.
    \item For $q \neq 2$, the set of almost good formal primes in degree $d$ which are bad formal immersion primes in characteristic $q$ is finite and effectively computable.
\end{enumerate}
\end{theorem}

This result allows us to define the integers $\BadFormalImmersion(d)$ and $\AGFI_d(q)$ that were already shown in \Cref{alg:type_1_primes}.

To prove \Cref{thm:explicit_formal_immersion} we first establish an upper bound on the bad formal immersion primes. Parent \cite[Theorem 1.8 and Proposition 1.9]{parent1999bornes} gives the general bound of $65(2d)^6$; however the following result allows us to reduce this considerably. To state it, we let $M \geq 3$ denote an odd integer, we represent cosets of $(\Z/M\Z)^\times$ by $a + M\Z$ with $a$ chosen to satisfy $0 < a < M$, and we define the map
\begin{align*}
    \eps_M : (\Z/M\Z)^\times &\to \left\{0,1\right\} \\
    a + M\Z &\mapsto
    \begin{cases*}
    0 & if $1 \leq a < M/2$\\
    1 & otherwise.
    \end{cases*}
\end{align*}
For a given $u \in (\Z/M\Z)^\times$ we may then define the matrix
\[ R_{d,u} := \left( \eps_M(na) - \eps_M(nu/a)\right)_{1 \leq n \leq d, \ a \in (\Z/M\Z)^\times},\]
with entries taken in $\Z$. This is a $d \times \phi(M)$ matrix, where $\phi$ denotes the Euler totient function.

\begin{proposition}\label{prop:formal_immersion_bound}
Let $d$ be a positive integer, $M \geq 3$ an odd integer, and $u \in (\Z/M\Z)^\times$. If the matrix $R_{d,u}$ has rank $d$ over $\Z$, then for all primes $p > 2Md$ such that $pu \equiv 1 \Mod{M}$, $p$ is a good formal immersion prime in degree $d$.
\end{proposition}

\begin{proof}
This is in fact a reformulation of Corollary 6.8 in \cite{derickx2019torsion}. Here one replaces the auxiliary prime $3$ with $q$ throughout Section 6 of \emph{loc. cit.}. The integer $M$ is chosen such that the matrix in the statement of Corollary 6.8 of \emph{loc. cit.} with entries taken in $\F_q$ has rank $d$ for all $u \in (\Z/M\Z)^\times$, so the images of $L_1\mathbf{e}, \ldots, L_d\mathbf{e}$ in $H_1(X_0(p)(\C),\F_q)$ are linearly independent over $\F_q$, which by Corollary 6.6 in \emph{loc. cit.} means that we have a formal immersion.
\end{proof}

We may therefore replace Parent's bound of $65(2d)^6$ with $2\hat{M}d$, where $\hat{M}$ is the smallest value of $M$ such that the matrix $R_{d,u}$ has rank $d$ \emph{for all} $u \in (\Z/M\Z)^\times$; this may be found algorithmically by trying successively larger odd values of $M$ until one is found which satisfies the rank conditions; the search is terminated at Parent's bound. This is implemented in the \path{construct_M} function in \path{sage_code/type_one_primes.py}.

Having now found a bound depending only on $d$ beyond which every prime $p$ is a good formal immersion prime, it follows that, for a given $d$, there are only finitely many primes for which we need to determine whether they are good, almost good, or bad. The next theorem of Parent - often referred to as \emph{Kamienny's criterion} - is the main ingredient for doing this; to state it, some notation is required. 
\begin{align*}
    \T &: \mbox{the Hecke algebra (\cite[Chapter II, Section 6]{Mazur3})}\\
    T_n &: \mbox{the $n$\textsuperscript{th} Hecke operator (\cite[Chapter II, Section 6]{Mazur3})}\\
    e &: \mbox{the winding element (\cite[Chapter II, Section 18]{Mazur3})}.
\end{align*}

\begin{theorem}[Parent, Theorem 4.18 in \cite{parent1999bornes}]\label{thm:formal_iff} 
Let $p$ and $q$ be distinct primes with $q \neq 2$. Then $f_p^{(d)}$ is a formal immersion at $q$ if and only if the modular symbols $T_1e,\ldots,T_de$ are linearly independent in $\mathbb T e/q\mathbb Te$.
\end{theorem}

Note that the condition $q \neq 2, p$ is not mentioned explicitly in \cite[Theorem 4.18]{parent1999bornes}, but it is a running assumption in that is mentioned earlier in the text.

We may now finish the proof of \Cref{thm:explicit_formal_immersion}.

\begin{proof}[Proof of \Cref{thm:explicit_formal_immersion}]
From \Cref{prop:formal_immersion_bound} and the definition of $\hat{M}$, we have that for $p > 2\hat{M}d$, $p$ is a good formal immersion prime; thus both of the sets in the statement of the Theorem are finite and bounded by $2\hat{M}d$, and it remains only to argue that they are effectively computable. We do this by showing that for a given prime $p$ one can exactly compute the primes $q \neq 2, p$ at which $f_p^{(d)}$ is not a formal immersion. 

By computing Hecke operators up to the Sturm bound one may compute a basis of $\mathbb T$ and hence a basis of $\mathbb T e$. One then defines a linear map $A: \Z^d \to \mathbb T e$ which sends the $i$\textsuperscript{th} basis element to $T_ie$. 

If $A$ is not injective, then for all primes $q$ the modular symbols $T_1e,\ldots,T_de$ are linearly dependent in $\mathbb T e/q\mathbb Te$, and hence $p$ is bad formal immersion prime by \Cref{thm:formal_iff}. 

If $A$ is injective, then $T_1e,\ldots,T_de$ are linearly independent over $\Z$, and the finitely many primes $q$ for which they become linearly dependent in $\mathbb T e/q\mathbb Te$ can be obtained from the Smith Normal Form of a matrix for $A$. In particular, $p$ is a good formal immersion prime if all nonzero coefficients of the Smith Normal Form are coprime to $2p$, and $p$ is an almost good formal immersion prime otherwise.
\end{proof}

\begin{remark}
Kamienny originally formulated his criterion (Proposition 3.1 in \cite{kamienny1992imrn}) in terms of the linear independence of Hecke operators in $\mathbb{T} \otimes \F_q$. This was subsequently developed by Merel, Oesterl\'{e} and Parent, and most recently generalised and made algorithmic by work of the second author with Kamienny, Stein and Stoll; the state-of-the-art for general modular curves $X_H$ is Proposition 5.3 of \cite{derickx2019torsion} which also works when $q=2$, but has the drawback of not giving an if and only if statement. Our algorithm only works for $X_0(p)$ and $q>2$, however it gives more precise information.
\end{remark}

The algorithm given by the proof of \Cref{thm:explicit_formal_immersion} is implemented in our package as the functions \path{R_dp}, \path{is_formal_immersion} and \path{bad_formal_immersion_data} in \path{sage_code/type_one_primes.py}. The results of the computation are stored in \path{bad_formal_immersion_data.json} as a method of caching these results to save time on subsequent runs of the algorithm for a given degree. (That is, this computation is only run once for each new $d$ encountered; thereafter the values are looked up in the JSON file.) \Cref{tab:bad_formal_immersion} shows the results of running this algorithm for degrees $d \leq 10$, where it was found that there were no almost good formal immersion primes (i.e. every prime $p \neq 2$ was either a good or a bad formal immersion prime in degree $d$).

\begin{table}[htp]
\begin{center}
\begin{tabular}{|c|c|c|}
\hline
$d$ & SGFIP & Sporadic bad FI primes\\
\hline
$2$ & $23$ & 37\\
$3$ & $41$ & 43, 73\\
$4$ & $47$ & 53, 61, 67, 73, 97\\
$5$ & $59$ & 61, 67, 73, 97\\
$6$ & $71$ & 73, 79, 83, 97, 103, 109, 113\\
$7$ & $101$ & 103, 107, 109, 113, 127, 137, 157\\
$8$ & $131$ & 137, 149, 157, 163, 193\\
$9$ & $131$ & 137, 139, 149, 151, 157, 163, 181, 193\\
$10$ & $167$ & 181, 193, 197, 211, 241\\
\hline
\end{tabular}
\vspace{0.3cm}
\caption{\label{tab:bad_formal_immersion}The good and bad formal immersion primes in degrees $d \leq 10$. The prime is good if it is larger than or equal to `SGFIP' (Smallest good formal immersion prime) and is not in `Sporadic bad FI primes'. For these degrees no almost bad formal immersion primes were found.}
\end{center}
\end{table}

\begin{remark}
We in fact ran the computation for all degrees up to $20$, and still found no almost good formal immersion primes. We can find no reason to explain this.
\end{remark}

We are now ready to present the proof of the main theorem of this section.

\begin{proof}[Proof of \Cref{thm:type_1_main}]
We take $q \neq p$ a rational prime, and $\fq | q$ an auxiliary prime. If $E$ has potentially good reduction at $\fq$, then by \Cref{cor:ABC_divs} $p$ will divide $C(\eps_0, \fq)$. If this integer were zero, then by definition of $C(\eps_0, \fq)$ there would exist a Frobenius root $\beta$ satisfying $\beta^{12h_\fq} = \gamma_\fq^{\eps_0}$, where $\gamma_\fq$ is a generator of the principal ideal $\fq^{h_\fq}$; but since $\eps_0 = (0,\ldots, 0)$, we deduce that $\beta^{12h_\fq} = 1$, which cannot happen (since $\beta$ has absolute value an integer multiple of $\sqrt{q}$); thus $C(\eps, \fq) \neq 0$, which gives a multiplicative bound on the primes $p$ for which $E$ has potentially good reduction at $\fq$, and explains the presence of the integer $C(\eps_0,\fq)$ in the integer $D(\Aux)$ in item (2) of \Cref{alg:type_1_primes}.

We are thus reduced to considering $E$ having potentially multiplicative reduction at $\fq$. That is, writing $x$ for the noncuspidal $K$-point on $X_0(p)$ corresponding to $E$, we have that $x$ reduces modulo $\fq$ to one of the two cusps $0$ or $\infty$.

Suppose first that $x$ reduces modulo $\fq$ to the zero cusp. This means that the kernel $W$ of the isogeny specialises to the identity component $(E_{/\F_{\fq}})^0$ of the reduction of $E$ at $\fq$. Since this coincides with the group scheme $\mu_p$ of $p$\textsuperscript{th} roots of unity over an (at most) quadratic extension of $\F_{\fq}$, we obtain an equality of groups between $W(\overline{K_\fq})$ and $\mu_p(\overline{K_\fq})$, where $K_\fq$ denotes the completion of $K$ at $\fq$. This precise case is considered by David in the proof of Proposition 3.3 of \cite{david2011borne}, and corresponds to the isogeny character satisfying $\lambda^2(\Frob_\fq) \equiv \Nm(\fq)^2 \Mod{p}$. However, since $\varepsilon = (0,\ldots,0)$ we obtain that $\lambda^{12h_\fq}$ acts trivially on $\Frob_\fq$, yielding the divisibility 
\begin{equation}\label{eqn:momose_corrected}
p |  \Nm(\fq)^{12h_\fq} - 1;
\end{equation}
i.e. that $p | B(\eps_0, \fq)$.

Suppose next that $x$ reduces modulo $\fq$ to the infinite cusp. We may then consider the point $x^\sigma$, for $\sigma \in \Sigma$ an embedding of $k$ in $K$, and observe that, if there is an $x^\sigma$ which specialises to the zero cusp modulo $\fq$, then we may apply the previous argument to the conjugate curve $E^\sigma$ (which also has isogeny signature $(0,\ldots,0)$) and conclude as before that $p | B(\eps_0, \fq)$.

We are thus reduced to considering the case that the $S$-section $(x^\sigma)_{\sigma \in \Sigma}$ (for $S := \Spec(\Z[1/p])$ as before) of $X_0(p)^{(d)}$ ``meets'' the section $\infty^d$ at $\fq$ (to use Kamienny's terminology in the discussion immediately preceding his Theorem 3.4 in \cite{kamienny1992torsion}; Mazur's language for this in the proof of Corollary 4.3 of \cite{mazur1978rational} is that these two sections ``cross'' at $\fq$). By applying Kamienny's extension of Mazur's formal immersion argument (Theorem 3.3 in \cite{kamienny1992imrn}), augmented with Merel's use of the winding quotient instead of the Eisenstein quotient, this implies that the map $f^{(d)}_{p}$ is \emph{not} a formal immersion along $\infty^d$ in characteristic $q$; therefore, either $p$ divides $\BadFormalImmersion(d)$, or $p$ divides $\AGFI_d(q)$. Since these integers are defined as products of primes, they are nonzero.
\end{proof}

\begin{remark}\label{rem:mom_mistake_type_1}
Note that Momose obtains the erroneous bound $p - 1 | 12h_k$ at the point where we obtain the multiplicative bound in \Cref{eqn:momose_corrected}. This occurs in the proof of Momose's Theorem 3 in \cite{momose1995isogenies}, and is arrived at through the claim that ``the restriction of $\lambda$ to the inertial subgroup $I_\fq$ of $\fq$ is $\pm \theta_p$'' ($\theta_p$ being Momose's notation for the mod-$p$ cyclotomic character). Momose cites the whole of \cite{delignerapoport} for this claim, which alas we were unable to find in those $174$ pages. If this claim were true, then, since we also know that $\lambda^2$ is unramified at $\fq$, we would obtain that $p - 1 | 2$, i.e., that $p = 2$ or $3$.
\end{remark}

\begin{remark}\label{rem:formal_immersion_at_2}
While \Cref{thm:explicit_formal_immersion} is enough to obtain a finite superset for the Type 1 primes, it is not making use of information that may be obtained from considering the auxiliary prime $q = 2$. The difficulty here is that there is no currently known lower bound (analogous to \Cref{prop:formal_immersion_bound}) on the primes $p$ beyond which $f_p^{(d)}$ is a formal immersion at $\infty^d$ in characteristic $2$.

Nevertheless, one may still employ Kamienny's criterion for various small values of $d$ and primes up to a reasonable bound and record the primes for which we \emph{do} have formal immersion in characteristic $2$. This was in fact already done for $3 \leq d \leq 7$ and for primes $11 \leq p \leq 2281$ in Lemma 5.4 of \cite{derickx2019torsion}, using Magma code which may be found on Stoll's webpage \cite{stollmagma}. Running this code also for $d = 2$ and slightly extending the bound on $p$ allows us to make use of the auxiliary prime $2$ to rule out possible Type one primes $p$ which are less than $2371$, unless they are in an explicitly computed set which is found in the file \path{formal_immersion_at_2.json}, which contains the hardcoded information, for each $2 \leq d \leq 7$ and each prime $13 \leq p \leq 2371$, whether or not we have formal immersion at $2$.
\end{remark}

\section{Momose Type 2 primes}\label{sec:momose_type_2}

In this section we consider isogenies whose signature is not only of Type 2, but whose isogeny character is furthermore of Momose Type 2. We start with a necessary condition on $p$ which must by satisfied by such isogenies. This generalises the condition (``Condition CC'') given for quadratic fields by the first named author in \cite{banwait2021explicit}.

\begin{proposition}\label{cond:CC}
Let $k$ be a number field, and $E/k$ an elliptic curve admitting a $k$-rational $p$-isogeny of Momose Type 2. Let $q$ be a rational prime, and $\fq$ a prime ideal of $k$ dividing $q$ of residue degree $f$ satisfying the following conditions:
\begin{enumerate}
    \item \label{item:q_f_less_p} $q^f < p/4$;
    \item $f$ is odd;
    \item $q^{2f} + q^f + 1 \not\equiv 0 \Mod{p}$.
\end{enumerate}
Then $q$ does not split in $\Q(\sqrt{-p})$.
\end{proposition}

\begin{proof}
If $E$ has potentially good reduction at $\fq$, then by the first and second assumption, we get from \cite[Lemma 5]{momose1995isogenies} that $q$ is inert in $\Q(\sqrt{-p})$.

We claim that if $E$ has potentially multiplicative reduction at $\fq$, and $q$ splits in $\Q(\sqrt{-p})$, then $q^{2f} + q^f + 1 \equiv 0 \Mod{p}$, contradicting assumption 3. 

To establish the claim, denote by $\lambda$ the isogeny character of $E$, assumed to be of Momose Type 2. From \cite[Lemma 3]{momose1995isogenies}, $\lambda$ is of the form $\psi\chi_p^{\frac{p+1}{4}}$ for a character $\psi$ of order dividing $6$.

This claim is then established as follows.
	\begin{align*}
	&\ \Nm(\fq) \equiv \psi(\Frob_\fq)^{\pm2} \Mod{p} \mbox{ (by \cite[Lemma 4]{momose1995isogenies}) }\\
	\Rightarrow &\ \Nm(\fq) \mbox{ is a 3\textsuperscript{rd} root of unity in } \F_p^\times \mbox{ (since $\psi$ has order dividing $6$ by \cite[Lemma 3]{momose1995isogenies}) }\\
	\Leftrightarrow &\ q^f \mbox{ is a 3\textsuperscript{rd} root of unity in } \F_p^\times \mbox{ (by definition of $f$) }\\
	\Leftrightarrow &\ q^{2f} + q^f + 1 \equiv 0 \Mod{p} \mbox{ (since (\ref{item:q_f_less_p}) implies $q^f \not\equiv 1 \Mod{p}$)}.
	\end{align*}
\end{proof}

\begin{remark}
See Lemma 5.1 of \cite{banwait2021explicit} for more details on the proofs of Momose's Lemmas 3, 4 and 5 in \cite{momose1995isogenies}.
\end{remark}

Next we provide a bound on the isogeny primes of Momose Type 2. This is the only result in our work which requires the Generalised Riemann Hypothesis.

\begin{proposition}\label{prop:type_2_bound}
Assume GRH. Let $k$ be a number field of degree $d$, and $E/k$ an elliptic curve possessing a $k$-rational $p$-isogeny, for $p$ a Type~2 prime. Then $p$ satisfies
\begin{equation}\label{eqn:type_2_finite}
p \leq (8d\log(12p) + 16\log(\Delta_k) + 10d + 6)^4.
\end{equation}
In particular, there are only finitely many primes $p$ as above.
\end{proposition}

\begin{proof}
In the proof of Theorem~6.4 of \cite{larson_vaintrob_2014}, the authors prove that a Momose Type~2 prime $p$ satisfies
\begin{equation}\label{eqn:type_2_ineq}
p \leq \left(1 + \sqrt{\Nm^k_\Q(v)} \right)^4
\end{equation}
where $v$ is a prime ideal of $k$ such that $v$ is split in $k(\sqrt{-p})$, is of degree $1$, does not lie over $3$, and satisfies the inequality
\[ \Nm_{k/\Q}(v) \leq c_7 \cdot (\log \Delta_{k(\sqrt{\pm p})} + n_{k(\sqrt{\pm p})} \log 3 )^2 \]
for an effectively computable absolute constant $c_7$ (note that they use $\ell$ instead of $p$, and $n_k$ denotes the degree of $k$).

The existence of such a $v$ follows from their Corollary 6.3, which requires GRH. Stepping into the proof of this Corollary, we arrive at a point where they apply the Effective Chebotarev Density Theorem to $\Gal(E'/k)$ - for a certain Galois extension $E'$ which fits into a tower of successive extensions $E'/E/k/\Q$ - to bound the norm of $v$ as
\[ \Nm_{k/\Q}(v) \leq c_5(\log \Delta_{E'})^2 \]
for an effectively computable absolute constant $c_5$.

However, we may at this breakpoint instead use Theorem 5.1 of \cite{bach1996explicit} on the Galois extension $E'/k$ of absolute degree $4d$ to obtain
\[ \Nm_{k/\Q}(v) \leq (4\log \Delta_{E'} + 10d + 5)^2. \]
This then subsequently (stepping back out into the proof of Theorem 6.4) yields the bound
\[ \Nm_{k/\Q}(v) \leq (16\log \Delta_{k} + 8d\log p + 8d\log 6 + 8d\log 2 + 10d + 5)^2. \]
Inserting this into \Cref{eqn:type_2_ineq} yields the result.
\end{proof}

Coupling \Cref{cond:CC} with the explicit bound on Type 2 primes given in \Cref{prop:type_2_bound}, we are able to algorithmically determine a superset for the Type 2 primes. 

\begin{algorithm}\label{alg:type_2_primes}
Given a number field $k$ of degree $d$, compute a set of primes $\TypeTwoPrimes(k)$ as follows.
\begin{enumerate}
    \item Initialise $\TypeTwoPrimes(k)$ to the empty set.
    \item Compute the bound $B_k$ on $p$ implied by \Cref{eqn:type_2_finite}.
    \item For $p \leq B_k$:
    \begin{enumerate}
        \item For each of the finitely many $q$ satisfying the conditions (1)-(3) of \Cref{cond:CC}:
        \begin{enumerate}
            \item If $q$ splits in $\Q(\sqrt{-p})$, then break this for-loop over $q$, and continue to the next $p$;
            \item else continue to the next $q$.
            \item If one has not continued to the next $p$ by this point, then append $p$ to $\TypeTwoPrimes(k)$ before continuing to the next $p$.
        \end{enumerate}
    \end{enumerate}
    \item Return  $\TypeTwoPrimes(k)$.
\end{enumerate}
\end{algorithm}

From the above discussion it is clear that the returned set is a superset for the primes $p$ for which there exists an elliptic curve over $k$ admitting a $k$-rational $p$-isogeny of Momose Type 2. This algorithm is implemented in the function \path{type_2_primes} in \path{sage_code/type_two_primes.py}, though as in \cite{banwait2021explicit} the bounds become rather large, so an optimised and parallelised implementation in PARI/GP may be found in \path{gp_scripts/partype2primes.gp}. 
Note that, by \Cref{prop:local_momose_1}, if $E$ is semistable, then signature Type $2$ does not arise, so one does not have Momose Type $2$ primes, and hence the algorithm to compute a superset for the \emph{semistable isogeny primes} is unconditional.

\section{Automatic Weeding}\label{sec:automatic_weeding}

With the ideas presented in the paper up to this point, we obtain, for each signature $\eps$, a nonzero integer $M(\eps)$ such that, if there is an elliptic curve over $k$ admitting a $k$-rational $p$-isogeny, then $p | M(\eps)$. We refer to $M(\eps)$ as the multiplicative upper bound on the isogeny primes of signature $\eps$, since primes dividing this integer are not necessarily actual isogeny primes for $k$.

This gives rise to the notion of \emph{weeding}, by which we mean techniques and conditions for ruling out possible isogeny primes for $k$. As part of our new software package \emph{Isogeny Primes}, we have automated many of these methods, which this section gives an overview of. These methods are all applied after a call to \path{prime_divisors} has been made, so one has in hand candidate isogeny primes $p$ at this point.

The top-level function which executes these automated weeding methods is \path{apply_weeding} in \path{sage_code/weeding.py}, called at the end of the routine.

\subsection{Filtering isogeny primes of signature $\eps$}\label{sec:filtering}
By combining the admissibility criterion \Cref{cor:admissible_signatures} with the necessary congruence conditions for the arising of the integers $4$, $6$ and $8$ in an isogeny signature, we obtain the following necessary conditions.

\begin{corollary}\label{cor:light_filter}
For each signature Type, \Cref{tab:eps_restrictions} gives necessary conditions that must be satisfied for the existence of a $k$-rational $p$-isogeny.
\end{corollary}

\begin{table}[htp]
\begin{center}
\begin{tabular}{|c|c|c|}
\hline
Type & $p$ must split or ramify? & Congruence conditions\\
\hline
Type 1 & NO & None\\
\hline
Quadratic nonconstant & YES & None\\
\hline
Sextic Constant & NO & $p \equiv 2 \Mod{3}$\\
\hline
Sextic nonconstant & YES & $p \equiv 2 \Mod{3}$\\
\hline
Type 2 & NO & $p \equiv 3 \Mod{4}$\\
\hline
Quartic nonconstant & YES & $p \equiv 3 \Mod{4}$\\
\hline
Mixed & YES & $p \equiv 1 \Mod{12}$\\
\hline
\end{tabular}
\vspace{0.3cm}
\caption{\label{tab:eps_restrictions}Necessary conditions for the existence of a $k$-rational $p$-sisogeny for each broad $\eps$-Type. `NO' is to be interpreted as `not necessarily'.}
\end{center}
\end{table}

These necessary conditions are implemented in \path{filter_ABC_primes} in \path{sage_code/common_utils.py}, and it explains the splitting of epsilons into the seven broad types expressed in \Cref{tab:eps_types}. We take the LCM of the multiplicative upper-bound integers for each $\eps$ of a given Type, and call \path{prime_divisors} on each of the resulting $7$ integers, allowing one to significantly reduce the number of calls to this potentially expensive function. This filtering always occurs by default in the algorithm.

\subsection{Isogeny character enumeration}\label{ssec:ice}
The grouping of all signatures into seven broad types, while possibly allowing for faster run-time, does come at the expense of possibly worse filtering. Knowing that a $p$-isogeny must arise from a given isogeny signature $\eps$ allows for finer necessary conditions, which we refer to as \textbf{isogeny character enumeration}.

This approach is based on the interpretation of a $p$-isogeny character $\lambda$ chiefly as a character of the group of ideals of $k$ coprime to $p$ (c.f. \Cref{rem:mu_galois_and_ideal}), which is achieved via class field theory. Since we need to be quite precise about this, we introduce the following notation to be used in addition to that set in \Cref{sec:notation_and_prelims}.

\begin{align*}
    \mc P &: \mbox{the set of primes of $k$ lying over $p$}\\
    \fm_p &: \prod_{\fp \in \mathcal P} \fp\\
    \mc I_k &: \mbox{the fractional ideals of $k$}\\
    \mc I_k^{\fm_p} &: \mbox{the fractional ideals of $k$ coprime to $\fm_p$}\\
    k^{\fm_p} &: \mbox{the units in $k^\times$ coprime to $\fm_p$}\\
    k^{\fm_p,1} &: \mbox{the subgroup of $a \in k^{\fm_p}$ satisfying $v_\fp(a -1) \geq v_\fp(\fm_p) = 1$ for $\fp \in \mathcal P$}\\
    \mc R_k^{\fm_p} \subseteq \mc I_k^{\fm_p} &: \mbox{the principal ideals generated by $k^{\fm_p}$}\\
    \mc R_k^{\fm_p,1} \subseteq \mc I_k^{\fm_p} &: \mbox{the principal ideals generated by $k^{\fm_p,1}$}\\
    \Cl_k^{\fm_p} &: \mc I_k^{\fm_p}/R_k^{\fm_p,1}, \ \mbox{the ray class group of modulus $\fm_p$}.
\end{align*}

Since the codomain of $\lambda^{12}$ is $\F_p^\times$ it follows that $\lambda^{12}$ is tamely ramified at all primes in $\mathcal P$, and we have seen from earlier that $\lambda^{12}$ is unramified at all primes outside of $\mathcal P$. Thus it is clear that $\fm_p$ is a modulus for $\lambda^{12}$, meaning that $\mc R_k^{\fm_p,1} \subseteq \ker \lambda^{12}$ and hence that $\lambda^{12} : \mc I_k^{\fm_p} \to \F_p^\times$ factors through $\Cl_k^{\fm_p}$. From the exact sequence
\[1 \to \mc R_k^{\fm_p}/\mc R_k^{\fm_p,1} \to \Cl_k^{\fm_p} \to \Cl_k \to 1,\]
together with the isomorphism $\mc R_k^{\fm_p}/\mc R_k^{\fm_p,1} \cong (\O_k/\fm_p)^\times/\O_k^\times$, \Cref{prop:momose_1_non_galois_md} can be interpreted as saying that $\eps$ classifies the possibilities of $\lambda^{12}|_{\mc R_k^{\fm_p}}$. In particular $\eps$ and $\fp_0$ together determine $\lambda^{12}|_{\mathcal{R}_k^{\fm_p}}$.

Now, given a prime $\fp_0$ in $\O_K$, and a signature $\eps$ one can define \begin{align*}
\chi_{\eps,\fp_0} : k^{\fm_p} &\to (\O_K/\fp_0)^\times \\
 \alpha &\mapsto \alpha^\eps \Mod{\fp_0}.
\end{align*}

There are two necessary conditions which the character $\chi_{\eps,\fp_0}$ should satisfy in order for it to come from the 12\textsuperscript{th} power of an isogeny character:
\begin{enumerate}
    \item $\chi_{\eps,\fp_0}(\O_k^{\times}) = 1$;
    \item $\im \chi_{\eps,\fp_0} \subseteq (\F_p^\times)^{12} \subseteq(\O_{K}/\fp_0)^\times$.
\end{enumerate}
The first condition comes from the fact that $\chi_{\eps,\fp_0}(\alpha)$ should only depend on the ideal that $\alpha$ generates, and the second one is clear since isogeny characters are $\F_p^\times$ valued.

Condition (1) is easy to computationally verify since one can just compute $\chi_{\eps,\fp_0}$ on a set of generators for $\O_k^\times$. Condition (2) can also be verified by observing that $\chi_{\eps,\fp_0}$ factors through the natural map
\[k^{\fm_p} \to (\O_k/{\fm_p})^\times \cong \prod_{\fp \in \mc P} (\O_{k_\fp}/\fp)^\times,\]
so it suffices to verify condition (2) on lifts of generators of each finite field $(\O_{k_\fp}/\fp)^\times$.

Once conditions (1) and (2) are satisfied by $\chi_{\eps,\fp_0}$ we can define $\chi'_{\eps,\fp_0}$ as the character $\chi'_{\eps,\fp_0} : \mc R_k^{\fm_p} \to (\F_p^\times)^{12}$ making the following diagram commute.

\centerline{\xymatrix{
  1 \ar[r] &\O_k^\times \ar[r] &k^{\fm_p} \ar[r]\ar[d]^{\chi_{\eps,\fp_0}} &\mc R_k^{\fm_p} \ar[r]\ar[d]^{\chi'_{\eps,\fp_0}} & 1\\
  & & (\O_K/\fm_p)^\times & (\F_p^\times)^{12} \ar@{_(->}[l]
}}
Thus, the question of whether $\chi_{\eps,\fp_0}$ arises as the 12\textsuperscript{th} power of an isogeny character is equivalent to whether $\chi'_{\eps,\fp_0}$ can be extended to a character  $\mu : \mc I_k^{\fm_p} \to (\F_p^\times)^{12}$.

\centerline{\xymatrix{
  1\ar[r] &\mc R_k^{\fm_p} \ar[r]\ar[d]_{\chi'_{\eps,\fp_0}} &\mc I_k^{\fm_p} \ar[r]\ar@{-->}[dl]^{\exists \mu} &\Cl_k \ar[r] & 1\\
  &  (\F_p^\times)^{12} 
}}

Since $\Cl_k$ is a finitely generated group this can be determined algorithmically as follows. Starting with a set of ideals $I_1,\ldots I_j \in \mc I_k^{\fm_p}$ which together generate $\Cl_k$  we iteratively compute a list $L_i$ that contains the different sequences of possible values $\mu(I_1),\ldots,\mu(I_i)$ for $i=1,\ldots,j$. Precisely, we implement the following algorithm.

\begin{algorithm}\label{alg:ice_1}
Given a number field $k$, a prime $p$, and isogeny signature $\eps$, return a set $L$ as follows.
\begin{enumerate}
    \item Compute ideals $I_1,\ldots I_j \in \mc I_k^{\fm_p}$ that together generate $\Cl_k$.
    \item Let $\Cl_0 := 1 \subset \Cl_k$ be the trivial subgroup, and for $i=1, \ldots, j$ let $\Cl_i \subset \Cl_k$ be the subgroup generated by $I_1,\ldots,I_i$.
    \item Let $I_0=(1)$ and $L_0=[[1]]$.
    \item For $i=1,...,j$ do \label{step4} \begin{enumerate}
    \item Let $L_i:=[]$ be the empty list;
    \item Compute $h_i$ as the order of $I_i$ in $\Cl_k/\Cl_{i-1}$;
    \item Write $I_i^{h_i} \sim \prod_{j=1}^{i-1} I_j^{e_{i,j}}$ in $\Cl_{i-1}$;
    \item Compute a generator $\alpha_i$ of the principal ideal  $I_i^{h_i} \prod_{j=1}^{i-1} I_j^{-e_{i,j}}$;
    \item For each sequence of possible values $c_0,c_1,\ldots,c_{i-1}\in (\F_p^\times)^{12}$ of $L_{i-1}$ do
    \begin{enumerate}
    \item Compute $c' := \chi_{\eps,\fp_0}(\alpha_i)\prod_{j=1}^{i-1}c_j^{e_{i,j}}$;
    \item Compute the set $\set{r_1,\ldots,r_{d_i}}$ of $h_i$\textsuperscript{th} roots of $c'$ in $(\F_p^\times)^{12}$;
    \item Append $[c_0,c_1,\ldots,c_{i-1},r_1]$ up to $[c_0,c_1,\ldots,c_{i-1},r_{d_i}]$ to $L_i$.
    \end{enumerate}
    \end{enumerate}
    \item Return $L = L_j$.
\end{enumerate}
\end{algorithm}

At the end of this algorithm $L$ will be a list of $(j+1)$-tuples corresponding to all possible extensions $\mu$ of $\chi'_{\eps,\fp_0}(\alpha_i)$ to $\mc I_k^{\fm_p}$ (for $i = 1,\ldots, j$); specifically, for $c = c_0, c_1,\ldots,c_j$, the extension $\mu_c$ corresponding to it is the character such that $\mu_c(I_i) = c_i$ for $i=1,\ldots,j$.

\begin{remark}
$c_0=1$ at every single iteration of each loop, so it can be left out. However the algorithm seems easier to understand with this initial condition.
\end{remark}

Given one such $j$-tuple $c_1,\ldots,c_j$ corresponding to an extension $\mu$ of $\chi'_{\eps,\fp_0}$, one may readily compute the value $\mu(I)$ for $I$ any ideal of $I_k^{\fm_p}$ as follows.

\begin{algorithm}\label{alg:ice_2}
Given a $j$-tuple $c_1,\ldots,c_j$ of elements of $(\F_p^\times)^{12}$ corresponding to an extension $\mu$ of $\chi'_{\eps,\fp_0}$, and an ideal $I$ of $I_k^{\fm_p}$, compute $\mu(I) \in (\F_p^\times)^{12}$ as follows.
\begin{enumerate}
    \item Write $I \sim \prod_{i=1}^{j} I_i^{e_{i}}$ in $\Cl_k$.
    \item Compute a generator $\alpha$ of the principal ideal  $I \prod_{i=1}^{j} I_i^{-e_{i}}$.
    \item Return $\mu(I) = \chi_{\eps,\fp_0}(\alpha)\prod_{i=1}^{j}c_i^{e_{i}}$.
\end{enumerate}
\end{algorithm}

Having computed all characters $\mu$ extending $\chi'_{\eps,\fp_0}$ we can try to prove that $\chi'_{\eps,\fp_0}$ does not come from the $12$\textsuperscript{th} power of an isogeny character by proving that any $\mu$ extending it does not come from an isogeny character. This is achieved by the following algorithm, which for an input prime ideal $\fq$ computes whether $\mu(\fq)$ is a 12\textsuperscript{th} power of a Frobenius root mod $p$.

\begin{algorithm}\label{alg:ice_3}
Given a $j$-tuple $c_1,\ldots,c_j$ of elements of $(\F_p^\times)^{12}$ corresponding to an extension $\mu$ of $\chi'_{\eps,\fp_0}$, and a prime ideal $\fq$ of $\O_k$ that is prime to $p$, return \path{True}/\path{False} as follows.
\begin{enumerate} 
    \item Compute $c = \mu(\fq)$ as in \Cref{alg:ice_2} .
    \item For all 12\textsuperscript{th} roots $r\in \F_p^\times$ of $c$ do
    \begin{enumerate}
        \item compute $\bar a_\fq :=r + \Nm(\fq)/r$;
        \item for all $a_\fq \in \Z$ with $a_\fq \equiv \bar a_\fq \Mod{p}$ and $|a_\fq| \leq 2\sqrt{\Nm(\fq)}$ do
        \begin{enumerate}
            \item if $x^2-a_\fq x+\Nm(\fq)$ is a Frobenius polynomial, return \path{True} and terminate;
        \end{enumerate}
    \end{enumerate}
    \item Return \path{False}.
\end{enumerate}
\end{algorithm}

Step (2.b.i) can be verified using \Cref{thm:waterhouse}.
Note that it makes sense to run this algorithm for all $\fq$ with $4\sqrt{\Nm(\fq)} < p$, since for these $\fq$ there is a chance that an $a_\fq$ as in step (b) does not exist, and hence that the algorithm returns \path{False}.

In the Type 1 case things work differently, since here we have to consider that $\mu$ is either trivial or $\chi_p^{12}$. Specifically, we apply the following.

\begin{algorithm}\label{alg:ice_4}
Given a $j$-tuple $c_1,\ldots,c_j$ of elements of $(\F_p^\times)^{12}$ corresponding to an extension $\mu$ of $\chi'_{\eps,\fp_0}$, return \path{True}/\path{False} as follows.
\begin{enumerate} 
        \item Set values = [].
        \item For all primes $\fq_1,\ldots,\fq_d | q\O_k$ do\begin{enumerate}
        \item Compute $\mu(\fq_i)$ as in \Cref{alg:ice_2}.
        \item If $\mu(\fq_i)$ is an $\fq_i$-Frobenius root mod $p$, return \path{False}.
        \item Append $\mu(\fq_i)$ to values.
    \end{enumerate}
    \item If values = $[1,1,\ldots,1]$ or $[\Nm(\fq_1)^{12},\Nm(\fq_2)^{12},\ldots,\Nm(\fq_d)^{12}]$, return \path{True}; else return \path{False}.
\end{enumerate}
\end{algorithm}

Note that if $\mu(\fq_i)$ at step (2.b) above is not $1$ or $\Nm(\fq_i) \Mod{p}$ then in fact $\mu$ cannot be an isogeny character.

Putting all of these parts together, we arrive at the main algorithm based on isogeny character enumeration, which is implemented in the function \path{character_enumeration_filter} in \path{sage_code/character_enumeration.py}, and which requires one to call to \path{prime_divisors} for each $\eps$.

\begin{algorithm}\label{alg:ice}
Given a number field $k$, a prime $p$, and isogeny signature $\eps$, return \path{True}/\path{False} as follows.
\begin{enumerate}
    \item Run \Cref{alg:ice_1}, and write $L$ for the output of it.
    \item If $\eps$ is not of signature Type $1$:
    \begin{enumerate}
    \item For each $j$-tuple in $L$, apply \Cref{alg:ice_2}. If any tuple returns \path{True}, return \path{True}; else return \path{False}.
    \end{enumerate}
    \item If $\eps$ is of signature Type $1$:
    \begin{enumerate}
    \item For each $j$-tuple in $L$, apply \Cref{alg:ice_4}. If any tuple returns \path{True}, return \path{True}; else return \path{False}.
    \end{enumerate}
\end{enumerate}
\end{algorithm}

\begin{remark}
The filtering arising from Isogeny character enumeration occurs as default in the implementation, but may be switched off via a command line argument by users who are experiencing performance issues, or who would like an answer more quickly. In practice we have found that calling \path{prime_divisors} on every $M(\eps)$ integer is not a significant bottleneck for number fields of degree at most 12.
\end{remark}

\subsection{The `No growth in minus part' method}

Lemma A.2 in the Appendix of \cite{banwait2021explicit} gives a general method to conclude that $X_0(p)(K) = X_0(p)(\Q)$ provided certain conditions are satisfied, the most crucial of which is that the Mordell-Weil group of the minus part $J_0(p)_{-}$ of the Jacobian $J_0(p)$ does not grow when base extending from $\Q$ to $K$. Indeed - while it is not needed in the current paper - the method there works \emph{mutatis mutandis} replacing $p$ with an arbitrary integer $N$.

Algorithmically checking this nongrowth condition for a general $K$ is at present not entirely straightforward; however in this section we illustrate how it may be checked in Sage in the simpler case that $K/\Q$ is an abelian extension of prime degree, using modular symbols computations. 

After constructing $J_0(p)_{-}$ as the kernel of $w_p + 1$ on the cuspidal subspace of modular symbols of level $\Gamma_0(p)$, one may count the size of the reduction $\widetilde{J_0(p)}_{-}(\F_\fp)$ at various primes $\fp$ of good reduction. One takes the GCD of the resulting point counts, and compares it to the size of the $\Q$-torsion subgroup of $J_0(p)_{-}$. Since all of the $\Q$-torsion of $J_0(p)$ is contained in $J_0(p)_{-}$ by Chapter 3, Corollary 1.5 of \cite{Mazur3}, this latter size is simply the numerator of $\left(\frac{p-1}{12}\right)$ by Theorem 1 of \emph{loc. cit.}. If the two quantities agree, then one can conclude that the torsion does not grow in the extension, and proceed to the rank check; although if they do not agree, then it is not necessarily the case that the torsion gets larger. In this case, since we are not sure, we abandon the routine.

The rank of $J_0(p)_{-}(K)$ is equal to the rank of $J_0(p)_{-}(\Q)$ plus the rank of the twist $J_0(p)_{-}^\chi(\Q)$, where $\chi$ is the Dirichlet character corresponding to $K$; this is where the assumption that $K/\Q$ is abelian of prime degree is used. One is therefore reduced to checking whether or not the rank of the $\chi$-twist is zero, which can be checked with a modular symbols computation involving the so-called \emph{$\chi$-twisted winding element}; see Section 2.2.2 of Bosman's\footnote{Johan Bosman was also a PhD student of Bas Edixhoven.} PhD thesis \cite{bosmanthesis} (this can also be found as Section 6.3.3 of \cite{couveignes2011computational}).

We therefore obtain the following algorithm which, if it returns \path{False}, then we may remove the possible isogeny prime $p$ from the final superset.

\begin{algorithm}\label{alg:method_of_appendix}
Given an abelian extension $K/\Q$ of prime degree $d$, and a prime $p$, return \path{True}/\path{False} as follows.

\begin{enumerate}
    \item If the class number of $\Q(\sqrt{-p})$ is either $1$ or $d$, return \path{True} and terminate.
    \item If the gonality of $X_0(p)$ is $\leq 2$, return \path{True} and terminate.
    \item Compute $\chi$, the Dirichlet character associated to $K$.
    \item If $d = 2$ and $\chi(p) = -1$, return \path{True} and terminate.  
    \item For rational primes $q \neq p$ up to a fixed bound $B$, compute the cardinality $|\widetilde{J_0(p)}_{-}(\F_\fq)|$ for $\fq | q$ (which must have residue field degree either $1$ or $d$), and take the GCD. If this does not equal the numerator of $ \left(\frac{p-1}{12}\right)$, return \path{True} and terminate.
    \item Compute the $\chi$-twisted winding element $e(\chi)$ associated to the modular symbols space of weight $2$ and level $p$.
    \item Compute the image of $e(\chi)$ under the rational period mapping associated to each of the factors in the Hecke decomposition of the cuspidal modular symbol space associated to $J_0(p)_{-}$. If any of these images are zero, return \path{True} and terminate.
    \item If not returned by this point, return \path{False}.
\end{enumerate}
\end{algorithm}

The implementation of this algorithm may be found in the method \path{works_method_of_appendix} in \path{sage_code/weeding.py}.

\subsection{Quadratic Weeding}

In the case that $K$ is a quadratic field, another method that was used previously by the first named author in determining $\IsogPrimeDeg(K)$ was the \textbf{\"{O}zman sieve}. This relied on the work of Bruin and Najman \cite{bruin2015hyperelliptic}, \"{O}zman and Siksek \cite{ozman2019quadratic}, and Box \cite{box2021quadratic} who determined the so-called exceptional quadratic points on small genus modular curves, together with earlier work of \"{O}zman in deciding on everywhere local solubility of twisted modular curves $X_0^d(N)$. 

This method has now been fully automated in the current package. The relevant information on quadratic points on modular curves has been encoded into the file \path{quadratic_points_catalogue.json}. For each candidate isogeny prime $p$, if the exceptional quadratic points on $X_0(p)$ have been determined, and none are rational over $K$, then the function \path{oezman_sieve(q,p)} is called for all ramified primes \path{q} in $K$ which furthermore are unramified in $\Q(\sqrt{-p}$); if any of these return \path{False}, then one has a local obstruction at $\Q_q$ to $\Q$-rational points on $X_0^d(p)$, where $K = \Q(\sqrt{d})$.

As the literature on the cataloguing of quadratic points on modular curves grows, these will be added to \path{quadratic_points_catalogue.json}, so over time our algorithm will improve. Having catalogues of all degree $d$ exceptional points on modular curves will likewise be of great benefit. 

One further weeding method employed in the case of a quadratic field $K$ is similar to the \"{O}zman sieve, and is based on \cite[Theorem 2.13]{najman2021splitting}. This applies in the case that $X_0(p)$ is hyperelliptic, and lists the primes which must be unramified in any quadratic field $K$ such that $X_0(p)$ admits a $K$-point which is not a $\Q$-point. These unramified primes have also been encoded into \path{quadratic_points_catalogue.json}. We refer to this method as the \textbf{Najman-Trbovi\'{c} filter}.

Putting these together we obtain the following algorithm which, if it returns \path{False}, then we may remove the possible isogeny prime $p$ from the returned superset.

\begin{algorithm}\label{alg:quadratic_weeding}
Given a quadratic field $K = \Q(\sqrt{d})$ and a prime $p$, return \path{True}/\path{False} as follows.

\begin{enumerate}
    \item If the exceptional quadratic points on $X_0(p)$ have all been determined:
    \begin{enumerate}
        \item If $X_0(p)$ admits an exceptional $K$-point, return \path{True} and terminate.
        \item Compute the primes $S$ which are ramified in $K$ but unramified in $\Q(\sqrt{-p})$.
        \item If \path{oezman_sieve(q,p)} returns \path{False} for any $q \in S$, return \path{False} and terminate.
    \end{enumerate}
    \item If $X_0(p)$ is hyperelliptic of genus $\geq 2$:
    \begin{enumerate}
        \item If any prime divisor $p$ of the discriminant of $K$ is listed as being an unramified prime by \cite[Theorem 2.13]{najman2021splitting}, return \path{False} and terminate.
    \end{enumerate}
    \item If not returned by this point, return \path{True}.
\end{enumerate}
\end{algorithm}

The function \path{apply_quadratic_weeding} in \path{sage_code/weeding.py} implements this algorithm.

\section{The combined Algorithm}\label{sec:combined}

In this section we collect all of the previously mentioned algorithms into one; this is the main algorithm which the package \emph{Isogeny Primes} implements, and serves as an overview of all of the other algorithms in the paper.

\begin{algorithm}\label{alg:combined}
Given a number field $k$ of degree $d$, compute a finite set of primes $S_k$ as follows.
\begin{enumerate}
    \item Initialise $S_k$ to the empty set.
    \item Run \Cref{alg:main} to produce $\MMIB(k)$. Append the prime divisors of $\MMIB(k)$ to $S_k$.
    \item Run \Cref{alg:type_1_primes} to produce $\TypeOneBound(k)$. Append the prime divisors of $\TypeOneBound(k)$ to $S_k$.
    \item Run \Cref{alg:type_2_primes} to produce $\TypeTwoPrimes(k)$. Add to $S_k$ the primes in $\TypeTwoPrimes(k)$.
    \item For each $p$ in $S_k$, apply the following filters. If any return False, then remove $p$ from $S_k$, and go to the next $p$.
        \begin{enumerate}
            \item Apply the congruence conditions expressed in \Cref{tab:eps_restrictions}.
            \item Apply Isogeny character enumeration (\Cref{alg:ice}).
            \item If $k$ is an abelian extension of $\Q$ of prime degree, then apply the `No growth in minus part' method (\Cref{alg:method_of_appendix}).
            \item If $k$ is a quadratic extension of $\Q$, then apply Quadratic weeding (\Cref{alg:quadratic_weeding}).
        \end{enumerate}
    \item Return  $S_k$.
\end{enumerate}
\end{algorithm}

We recall the main result (\Cref{thm:combined} from the Introduction) concerning the significance of the above algorithm for the determination of $\IsogPrimeDeg(k)$, and provide a summary proof.

\begin{theorem}
Let $k$ be a number field. Then \Cref{alg:combined} outputs a finite set of primes $S_k$ such that, if $p$ is an isogeny prime for $k$ whose associated isogeny character is not of Momose Type 3, then, conditional on GRH, $p \in S_k$. In particular,
\begin{enumerate}
    \item if $k$ does not contain the Hilbert class field of an imaginary quadratic field, then $S_k$ contains $\IsogPrimeDeg(k)$;
    \item the above results are unconditional for the restricted set of semistable isogeny primes for $k$.
\end{enumerate}
\end{theorem}

\begin{proof}
That $S_k$ is finite follows from the non-zero-ness of the integers $\MMIB(k)$ (\Cref{thm:main-brief}) and $\TypeOneBound(k)$ (\Cref{thm:type_1_brief}), and the finiteness of the set $\TypeTwoPrimes(k)$. That it is (conditional upon GRH) a superset for the isogeny primes for $k$ which are not of Momose Type 3 follows from the following facts:

\begin{enumerate}
    \item $\MMIB(k)$ is a multiplicative bound on isogeny primes which are not of Momose Types 1, 2 or 3 (\Cref{thm:main-brief});
    \item $\TypeOneBound(k)$ is a multiplicative bound on isogeny primes which are of Momose Type 1 (\Cref{thm:type_1_brief});
    \item $\TypeTwoPrimes(k)$ is, conditional upon GRH, a superset for the isogeny primes which are of Momose Type 2 (\Cref{thm:type_2_brief}).
\end{enumerate}

Item (1) follows from the definition of Momose Type 3 isogenies (\Cref{def:momose_type_3}), which requires $k$ to contain the Hilbert class field of an imaginary quadratic field, and Item (2) follows from the observation (\Cref{prop:local_momose_1}) that semistable elliptic curves cannot possess isogenies of Signature Type $2$, and hence \emph{a fortiori} not of Momose Type $2$; we therefore unconditionally have that $\TypeTwoPrimes(k)$ is empty.
\end{proof}

\section{Degree three points on \texorpdfstring{$X_0(p)$}{X0(p)}}\label{sec:cubic_examples}

In this next section we will determine $\IsogPrimeDeg(K)$ for various cubic fields $K$, and prove \Cref{thm:first_cubic}. This requires a study of cubic points on the modular curves $X_0(p)$ for small primes $p$. Such a study has recently been carried out by Box,  Gajovi\'c and Goodman \cite{box2021cubic}, who explicitly determine all cubic points on $X_0(p)$ for the primes $p=53, 61, 67$ and $73$. These primes are such that $J_0(p)$ has positive rank, and the authors of \emph{loc. cit.} skip the relatively easy cases of where $J_0(p)$ is of rank $0$. We therefore augment their work with the following two results: \Cref{thm:genus_2_cubic} which deals with the genus $2$ cases, and \Cref{thm:rank_0_cubic}, which deals with the higher genus cases. Note that for all $p$ considered below, $X_0(p)$ is hyperelliptic.

\begin{theorem}\label{thm:genus_2_cubic}
Suppose $p = 23, 29$ or $31$. Let $K$ be a cubic field such that $X_0(p)$ admits a noncuspidal $K$-rational point. Then we have the following.

\begin{enumerate}
    \item The discriminant $\Delta_K$ of $K$ is negative.
    \item There is a finite set of explicitly computable hyperelliptic curves over $\Q$ of genera $2$ or $3$ such that the quadratic twist at $\Delta_K$ of at least one of them admits a $\Q$-rational point.
\end{enumerate}
\end{theorem}

\begin{proof}
We use the usual strategy for studying degree $3$ points by studying $X_0(p)^{(3)}(\Q)$ and it's canonical map to $\Pic^3 X_0(p)(\Q)$. For each of these 3 values of $p$ the curve $X_0(p)$ is hyperelliptic with the Atkin-Lehner involution being the hyperelliptic involution. Additionally $J_0(p)$ is of rank $0$. This implies that $$J_0(p)(\Q) = J_0(p)(\Q)_{tors} =  \Z/N\Z(0-\infty) $$ with $N = {\mathrm{Numerator}}(\frac{p-1}{12})$. In particular if $a$ runs over all residue classes in $\Z/N\Z$ then $3\infty + a(0-\infty)$ runs over a complete set of representatives of $\Pic^3 X_0(p)(\Q)$, meaning that every divisor of degree $3$ is linearly equivalent to one of the form $3\infty + a(0-\infty)$. 

Consider the associated Riemann-Roch spaces $\mathcal{L}(3\infty + a(0-\infty))$ for all $a \in \Z/N\Z$. Note that these are all $2$-dimensional, since the genus of these curves is $2$, so the Riemann-Roch theorem for a degree $3$ divisor $D$ gives  $\dim H_0(X,D) = \dim H_0(X,D) - \dim H_0(X,K-D) = \deg D + 1-g = 2$. In particular, this shows that every effective divisor of degree $3$ is contained in exactly one $1$-dimensional family parameterized by $\P^1$.

By following the techniques in work of the second author with Najman (\cite{derickx2019torsionnajman}, see particularly the proofs of Lemmas 4.9 and 4.11) one may explicitly compute defining equations for each of these finitely many families. Briefly, one constructs models $f(x,t)=0$ for $X_0(p)$ whose projections to the $t$-coordinate coincides with the degree 3 map to $\P^1$ induced by the global sections of $\mathcal{L}a(3\infty + a(0-\infty))$ for each $a \in \Z/N\Z$. The cubic field corresponding to a value of $t=b \in \Q$ is $K_b=\Q[x]/f(x,b)$. Hence the discriminant of $K$ differs from the discriminant of $f$ by a square in $\Q$. Write $g(t) = \Delta_x f(x,t)$ where $\Delta_x$ means taking the discriminant with respect to the variable $x$. Then in particular we have 
\begin{equation}\label{eqn:quad_twists}
\Delta(K_b)y^2=g(b)
\end{equation}
for some value of $y$ in $\Q$. We checked that $g(x) \leq 0$ proving that $\Delta(K_b) \leq 0$ and hence (1). For (2), we computed the genera of the curves $\mathcal{C} : y^2 = g(x)$ and checked that they were all of genus $2$ or $3$. Since \Cref{eqn:quad_twists} gives the quadratic twist of $\mathcal{C}$, this completes the proof of (2). The computations described in this proof are carried out in \path{magma_scripts/Genus2Cubic.m}.
\end{proof}

Since totally real cubic fields must have positive discriminant, and cyclic cubic fields must always be totally real, we immediately obtain the following result.

\begin{corollary}\label{cor:cyclic_cubic}
Suppose $p = 23, 29$ or $31$, and let $K$ be a totally real cubic field. Then $X_0(p)(K)$ consists only of the two $\Q$-rational cusps.
\end{corollary}

\begin{remark}
The hyperelliptic curves constructed in the proof of \Cref{thm:genus_2_cubic} are computed in the main function of \path{magma_scripts/Genus2Cubic.m}. For $p = 23, 29$ and $31$, there are respectively $9$, $5$ and $3$ such curves. For $p = 31$, since there are only three of them, we list them here:
{\small
\begin{align*}
    \mathcal{C}_1 : y^2 &= -3x^8 - 8x^7 + 392x^6 + 992x^5 - 32976x^4 - 286336x^3 - 980352x^2 - 1560064x - 963328\\
    \mathcal{C}_2 : y^2 &= -3x^6 + 104x^5 - 1400x^4 + 9040x^3 - 27696x^2
        + 34880x - 22208\\
    \mathcal{C}_3 : y^2 &= -108x^6 + 2552x^5 - 12276x^4 - 64944x^3 - 
        124436x^2 - 134728x - 66188.\\
\end{align*}
}
\end{remark}

\begin{theorem}\label{thm:rank_0_cubic}
Suppose $p = 41, 47, 59$ or $71$; then $X_0(p)$ has $2,2,1$ and $0$ degree $3$ points respectively. Furthermore the degree 3 number fields over which these points are defined are as follows:
\begin{center}
\begin{minipage}{0.65\textwidth}
\begin{itemize}
    \item[$p=41$:] $\Q[x]/(x^3 - x^2 + x + 2)$ of discriminant $-139$
    \item[$p=47$:] $\Q[x]/(x^3 + x^2 + 2x + 12)$ of discriminant $-883$
    \item[$p=59$:] $\Q[x]/(x^3 - x^2 - x + 2)$ of discriminant $-59$.
\end{itemize}
\end{minipage}
\end{center}
\end{theorem}
\begin{proof}
The proof proceeds as in the proof of \Cref{thm:genus_2_cubic}, although now it is no longer true that $\mathcal{L}(3\infty + a(0-\infty))$ has dimension $2$ for all $a \in \Z/N\Z$.

We proceed via explicit Magma computation. For each value of $p$ we found exactly two values of $a$ where $\mathcal{L}(3\infty + a(0-\infty))$ was 2 dimensional; these two cases correspond to the 1 dimensional families of degree $3$ divisor of the form $\infty + D$ and $0 + D$ where $D$ is a divisor in the hyperelliptic class. Hence these two families don't correspond to degree $3$ points. 

In each of the remaining cases $\mathcal{L}(3\infty + a(0-\infty))$ was either $0$ or $1$ dimensional. In the $0$ dimensional cases there is nothing to do. For the 1 dimensional cases we explicitly found the unique effective divisor of degree $3$ linearly equivalent to $3\infty + a(0-\infty)$. In the cases this divisor was irreducible we explicitly computed a defining polynomial for the residue field of this divisor, this gives the classification. The code carrying out these explicit computations is found in \path{magma_scripts/HigherGenusCubic.m}.
\end{proof}

\begin{corollary}\label{cor:cubic_pos_disc}
Let $K$ be a cubic field of positive discriminant. If there exists an elliptic curve over $K$ admitting a $K$-rational $p$-isogeny, for $p$ not in $\IsogPrimeDeg(\Q)$, then $p \geq 79$.
\end{corollary}

\begin{proof}
We combine \Cref{thm:genus_2_cubic} and \Cref{thm:rank_0_cubic} with Theorem 1.1 of \cite{box2021cubic}, and check that the finitely many cubic fields given in Section 5.2 of \cite{box2021cubic} all have negative discriminant.
\end{proof}

By combining the results of this section with the program \emph{Isogeny Primes} explained in this paper, we determine $\IsogPrimeDeg(K)$ for some cubic fields $K$, thereby proving \Cref{thm:first_cubic}.

\begin{theorem}
Assume GRH. For each number field in \Cref{tab:cubic_isogenies}, the column `New Isogeny Primes' lists the isogeny primes for that number field which are not in $\IsogPrimeDeg(\Q)$.
\end{theorem}

\begin{table}[htp]
\begin{center}
\begin{tabular}{|c|c|c|c|}
\hline
$\Delta_K$ & $f_K$ & LMFDB Label & New Isogeny Primes\\
\hline
$49$ & $x^3 - x^2 - 2x + 1$ & \href{https://www.lmfdb.org/NumberField/3.3.49.1}{3.3.49.1} & --\\
$148$ & $x^{3} - x^{2} - 3x + 1$ & \href{https://www.lmfdb.org/NumberField/3.3.148.1}{3.3.148.1} & --\\
$-2891$ & $x^{3} - x^2 - 2x - 20$ & \href{https://www.lmfdb.org/NumberField/3.1.2891.3}{3.1.2891.3} & 29\\
\hline
\end{tabular}
\vspace{0.3cm}
\caption{\label{tab:cubic_isogenies}Determination of $\IsogPrimeDeg(K)$ for some cubic number fields $K$.}
\end{center}
\end{table}

\begin{proof}
Running \emph{Isogeny primes} on the first two number fields in \Cref{tab:cubic_isogenies} shows that the largest possible isogeny prime is at most $73$, so we conclude with \Cref{cor:cubic_pos_disc}. For the last cubic field $K$ in the Table, we run \emph{Isogeny primes} to obtain a superset as follows:
\[ \IsogPrimeDeg(K) \subseteq \IsogPrimeDeg(\Q) \cup \left\{23, 29, 31, 73\right\}.\]
From \cite{box2021cubic} we know that $73$ is not an isogeny prime for this cubic field. A search for $K$-rational points on the modular curve $X_0(29)$ in Magma reveals that there are elliptic curves admitting a $K$-rational $29$-isogeny. The primes $23$ and $31$ are ruled out with part (2) of \Cref{thm:genus_2_cubic}. We computed the finite set of hyperelliptic curves, and showed that the quadratic twist by the squarefree part of $-2891$ (i.e. $-59$) of each of these curves was not everywhere locally soluble, thereby showing that none of them admit a $\Q$-point.
\end{proof}

\bibliographystyle{alpha}
\bibliography{main.bib}{}
\end{document}